\documentclass[11pt]{amsart}
\usepackage{amsthm,amsfonts,amssymb,amscd,euscript,mathrsfs,bbm}
\usepackage[dvips]{graphicx}

\usepackage{color}
\usepackage{listings}
\usepackage{graphics}
\usepackage{tikz}
\usepackage[draft]{todonotes}
\usepackage{comment}
\input{xy} 
\xyoption{all}

\usepackage[colorlinks=true, pdfstartview=FitV,linkcolor=black,citecolor=black,urlcolor=black]{hyperref}
 
\newenvironment{red}{\relax\color{red}}{\relax}
\newenvironment{blue}{\relax\color{blue}}{\hspace*{.5ex}\relax}
\newcommand{\ber}{\begin{red}}
\newcommand{\er}{\end{red}}
\newcommand{\beb}{\begin{blue}}
\newcommand{\eb}{\end{blue}}
\newcommand{\GS}{Gr\"obner--Shirshov }

\newcommand{\bc}{{\boldsymbol{c}}}

\theoremstyle{plain}
\newtheorem{thm}{Theorem}[section]
\newtheorem{lem}[thm]{Lemma}
\newtheorem{prop}[thm]{Proposition}
\newtheorem{cor}[thm]{Corollary}

\theoremstyle{definition}
\newtheorem{definition}[thm]{Definition}
\newtheorem{example}[thm]{Example}
\newtheorem{remark}[thm]{Remark}
\numberwithin{equation}{section} \numberwithin{figure}{section}
\numberwithin{table}{section}

\begin{document}

\title[Rigid reflections and reduced roots]{Rigid reflections of rank 3 Coxeter groups and \\ reduced roots of rank 2 Kac--Moody algebras}

\author[K.-H. Lee]{Kyu-Hwan Lee$^{\star}$}
\thanks{$^{\star}$This work was partially supported by a grant from the Simons Foundation (\#712100).}
\address{Department of
Mathematics, University of Connecticut, Storrs, CT 06269, U.S.A.}
\email{khlee@math.uconn.edu}

\author[J. Yu]{Jeongwoo Yu$^{\dagger}$$^{\diamond}$}
\thanks{$^{\dagger}$ This work is partially supported by the National Research Foundation of Korea(NRF) grant funded by the Korea government(MSIT) (No.\,2019R1A2C1084833 and 2020R1A5A1016126).}
\thanks{$^{\diamond}$ Corresponding Author: Jeongwoo Yu}
\address{Department of Mathematical Sciences, Seoul National University, Seoul 08826,
South Korea}
\email{ycw453@snu.ac.kr}
\begin{abstract}
In a recent paper by K.-H. Lee and K. Lee, rigid reflections are defined for any Coxeter group via non-self-intersecting curves on a Riemann surface with labeled curves. When the Coxeter group arises from an acyclic quiver, the rigid reflections are related to the rigid representations of the quiver. 
For a family of rank $3$ Coxeter groups, it was conjectured in the same paper that there is a natural bijection from  the set of reduced positive roots of a symmetric rank $2$ Kac--Moody algebra onto the set of rigid reflections of the corresponding rank $3$ Coxeter group. In this paper, we prove the conjecture.
\end{abstract}

\maketitle



\noindent {\bf Data availability statement}: Data sharing not applicable to this article as no datasets were generated or analysed during the current study.

\section{Introduction}

Let $Q$ be an acyclic quiver of rank $n$, and $\mathrm{mod}(Q)$ be the category of finite dimensional representations of $Q$. In order to understand the category $\mathrm{mod}(Q)$, one needs to consider the indecomposable representations without self-extensions, called {\em rigid} representations. Their dimension vectors form a special subset of the set of positive real roots of the Kac--Moody algebra $\mathfrak g(Q)$ associated to $Q$, and are  called {\em real Schur roots}. These roots also appear in the denominators of cluster variables, or as the $c$-vectors  of the  cluster algebra associated to $Q$, and can be described combinatorially in terms of non-crossing partitions. See \cite{BDSW,CK,Ch,C-B,HK,IS,S,Se,ST} for more details on these connections. 

As a new geometric/combinatorial approach to describe rigid representations and real Schur roots, K.-H. Lee and K. Lee conjectured in their paper \cite{LL}  a correspondence between rigid representations in $\mathrm{mod}(Q)$ and the set of certain non-self-intersecting
curves on a Riemann surface $\Sigma$ with $n$ labeled curves. The conjecture is now proven by A. Felikson and P. Tumarkin \cite{FT} for acyclic quivers with multiple edges between every pair of vertices. Very recently, S. D. Nguyen \cite{Ngu} informed us that he proved the conjecture for an arbitrary acyclic quiver.

The conjecture actually characterizes the family of reflections in the Weyl group of $\mathfrak g(Q)$ which are associated to real Schur roots  via  non-self-intersecting curves in $\Sigma$. Since  reflections make sense for any Coxeter groups, the geometric characterization can be carried over. Indeed, the {\em rigid reflections} are defined in \cite{LL1} for any Coxeter group $W$ to be those corresponding to non-self-intersecting curves in $\Sigma$. Unexpectedly, an interesting  phenomenon was observed that the rigid reflections of $W$ are parametrized by the positive roots of a seemingly unrelated Kac--Moody algebra $\mathcal H$, and the phenomenon was investigated in detail for a family of rank 3 Coxeter groups.  

To be precise, for each positive integer $m\geq 2$, consider the following Coxeter group
 $$W(m)=\langle s_1,s_2,s_3 \ : \   s_1^2=s_2^2=s_3^2=(s_1s_2)^m=(s_2s_3)^m=e \rangle.
 $$
Let  $\mathcal H(m)$ be the rank 2 Kac--Moody algebra associated with the Cartan matrix ${\tiny \begin{pmatrix}  2&-m\\-m&2\end{pmatrix}}$. Denote an element of the root lattice of $\mathcal H(m)$ by $[a,b]$, $a, b \in \mathbb Z$, where $[1,0]$ and $[0,1]$ are the positive simple roots. 
A root $[a,b]$ of $\mathcal H(m)$ is called {\em reduced} if $\gcd(a,b)=1$ and $ab\neq 0$. Note that the set of reduced roots includes both real and imaginary roots. A reduced root determines a non-self-intersecting curve $\eta$ on the torus $\Sigma$ with triangulation by three labeled curves. 

Now define a function, $[a,b] \mapsto s([a,b]) \in W(m)$, by reading off the labels of the intersection points of $\eta$ with the labeled curves on $\Sigma$ and by writing down the products of simple reflections accordingly. See \eqref{pic-ab} for an example.  
In \cite{LL1}, it was conjectured that this function $[a,b] \mapsto s([a,b])$ is a bijection from the set of reduced roots of $\mathcal H(m)$ onto the set of rigid reflections of $W(m)$. If established, it would show that the set of rigid reflections in $W(m)$ has a structure coming from the set of reduced roots of $\mathcal H(m)$. Most importantly, the Weyl group action on the set of roots of $\mathcal H(m)$ would be transported to the set of rigid reflections on $W(m)$. In the same paper \cite{LL1}, as a main result, it was shown that the function is surjective; however, injectivity was checked only for $m=2$.  

One of the main difficulties in showing injectivity is directly related to the word problem for $W(m)$. Since $s([a,b])$ are given as words in simple reflections, one needs to determine when such two words represent the same (or different)  elements in $W(m)$. A solution to this problem may be given by an algorithm to write $s([a,b])$ into a canonical form or a standard word.   

\medskip

In this paper, we obtain such a reduction algorithm and prove the conjecture of \cite{LL1}.

\begin{thm} \label{thm-in}
For $m \ge 2$, the function, $[a,b] \mapsto s([a,b])$, is a bijection from the set of reduced positive roots of $\mathcal H(m)$ onto  the set of rigid reflections of $W(m)$. 
\end{thm}

The canonical forms or standard words of the elements in $W(m)$ are determined by applying \GS basis theory. In the first substantial step, {\em canonical sequences} of positive integers are assigned to each $[a,b]$ and a reduction is accomplished accordingly. The result is described in Corollary \ref{cor-red}. This reduction through canonical sequences can be visualized naturally in terms of associated curves on the torus $\Sigma$ and are related to the aforementioned Weyl group action on the set of rigid reflections in $W(m)$. Moreover, we note that the canonical sequences was used in a study of $2$-bridge link groups \cite{LS} in a slightly different form.     

However, this reduction through canonical sequences are not sufficient for our purpose, and we need to perform further reduction until we obtain standard words for the elements in $W(m)$ to distinguish them explicitly and to show injectivity of the map $[a,b] \mapsto s([a,b])$.

As is well known, the word problem for a group is intractable, in general. Surprisingly, \GS bases for $W(m)$ are not much different from the original set of defining relations, though $W(m)$ is infinite, and our reduction process becomes feasible. It would be interesting to see if there are other families of infinite Coxeter groups with relatively simple \GS bases. To such families, the method of this paper will generalize to reveal precise connections between rigid reflections and roots of Kac--Moody algebras. 

\medskip

The organization of this paper is as follows. In Section \ref{rigid}, we recall the definition of rigid reflections. In the next section, we consider the rank 3 Coxeter groups $W(m)$ and collect known results from \cite{LL1} about the rigid reflections of $W(m)$. In Section \ref{can sequences}, the canonical sequence and level of a reduced root $[a,b]$ are defined and their properties are studied. Using the canonical sequence, we achieve a substantial reduction of $s([a,b])$. In Section \ref{reductions}, we complete the reduction process to obtain the standard words for the elements of $W(m)$ and prove Theorem \ref{thm-in}, up to computation of \GS bases for $W(m)$ which is accomplished in the last section. 

\subsection*{Acknowledgments} We are grateful to Jae-Hoon Kwon and Kyungyong Lee for helpful discussions. We thank the anonymous referee whose useful comments helped to improve the exposition of this paper. 

\medskip

\section{Rigid reflections} \label{rigid}
In this section  we recall the definition of a rigid reflection from \cite{LL1}.

\medskip

Let $$
W=\langle s_1,s_2,...,s_n \ : \   s_1^2=\cdots=s_n^2=e, \ (s_is_j)^{m_{ij}}=e \rangle
$$ 
be a Coxeter group with $m_{ij}\in\{2,3,4,...\} \cup \{ \infty \}$. In order to define the rigid reflections of $W$, we introduce\footnote{One can use a different, equivalent geometric model as in \cite{FT, Ngu}.}  a Riemann surface  $\Sigma$ equipped with $n$ labeled curves as below.
Let $G_1$ and $G_2$ be two identical copies of a regular $n$-gon. Label the edges of each of the two  $n$-gons
 by $T_{1}, T_{2}, \dots , T_{n}$ counter-clockwise. On $G_i$ $(i=1,2)$, let $L_i$ be the line segment from the center of $G_i$ to the common endpoint of  $T_{n}$ and $T_{1}$. Later,  these line segments will only be used  to designate the end points of admissible curves and will not be used elsewhere.   Fix the orientation of every edge of $G_1$ (resp.  $G_2$) to be 
 counter-clockwise (resp. clockwise) as in the following picture. 
 
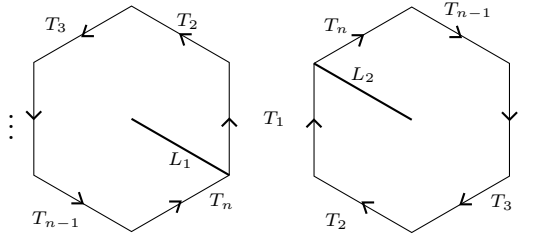
\begin{center}
 \begin{tikzpicture}[scale=0.5]
\node at (2.4,-2.2){\tiny{$T_n$}};
\node at (1.5,2.6){\tiny{$T_2$}};
\node at (3.8,0){\tiny{$T_1$}};
\node at (-2.0,-2.7){\tiny{$T_{n-1}$}};
\node at (-2.0,2.5){\tiny{$T_3$}};
\node at (-3.2,0){\vdots};
\draw (0,0) +(30:3cm) -- +(90:3cm) -- +(150:3cm) -- +(210:3cm) --
+(270:3cm) -- +(330:3cm) -- cycle;
\draw [thick] (2.4,-0.2) -- (2.6,0)--(2.8,-0.2);
\draw [thick] (1.4,1.95) -- (1.3,2.25)--(1.6,2.25);  
\draw [thick] (-1.0,2.2) -- (-1.3,2.2)--(-1.2,2.5);  
\draw [thick] (-2.4,0.2) -- (-2.6,0)--(-2.8,0.2);   
\draw [thick] (1.0,-2.2) -- (1.3,-2.2)--(1.2,-2.5);  
\draw [thick] (-1.4,-1.95) -- (-1.3,-2.25)--(-1.6,-2.25);  
\draw [thick] (0,0)--(2.6,-1.5);  
\node at (1.3,-1.1){\tiny{$L_1$}};
\end{tikzpicture}
\begin{tikzpicture}[scale=0.5]
\node at (-1.3,1.2){\tiny{$L_2$}};
\draw [thick] (0,0)--(-2.6,1.5);  
\node at (2.4,-2.2){\tiny{$T_3$}};
\node at (1.5,2.9){\tiny{$T_{n-1}$}};
\node at (3.3,0){\vdots};
\node at (-2.0,-2.7){\tiny{$T_2$}};
\node at (-2.0,2.5){\tiny{$T_n$}};
\draw (0,0) +(30:3cm) -- +(90:3cm) -- +(150:3cm) -- +(210:3cm) --
+(270:3cm) -- +(330:3cm) -- cycle;
\draw [thick] (-2.4,-0.2) -- (-2.6,0)--(-2.8,-0.2);
\draw [thick] (-1.4,1.95) -- (-1.3,2.25)--(-1.6,2.25);  
\draw [thick] (1.0,2.2) -- (1.3,2.2)--(1.2,2.5);  
\draw [thick] (2.4,0.2) -- (2.6,0)--(2.8,0.2);   
\draw [thick] (-1.0,-2.2) -- (-1.3,-2.2)--(-1.2,-2.5);  
\draw [thick] (1.4,-1.95) -- (1.3,-2.25)--(1.6,-2.25);  
\end{tikzpicture}
 \end{center}

 Let $\Sigma$ be the Riemann surface of genus $\lfloor \frac{n-1}{2}\rfloor$
obtained by gluing together the two $n$-gons with all the edges of the same label identified according 
to their orientations.  The edges of the $n$-gons become $n$ different curves in $\Sigma$. If $n$ is odd, all the vertices of the two $n$-gons 
are identified to become one point in $\Sigma$ and the curves obtained from the edges become loops. If $n$ is even, two distinct
 vertices are shared by all curves. Let $\mathcal{T}={T}_1\cup\cdots{T}_n\subset \Sigma$, and $V$ be the set of the vertex (or vertices) on $\mathcal{T}$.  

 Let $\mathfrak W$ be the set of words from the alphabet $\{1,2,...,n\}$, and let $\mathfrak R\subset\mathfrak W$ be the subset of words $\mathfrak w=i_1i_2 \cdots i_k$ such that $k$ is an odd integer and $i_{j}=i_{k+1-j}$ for all $j\in\{1,...,k\}$, in other words, $s_{i_1}s_{i_2} \cdots s_{i_k}$ is a reflection in $W$. For  $\mathfrak{w}=i_1i_2 \cdots i_k\in \mathfrak W$, denote $s_{i_1}s_{i_2}...s_{i_k}\in W$ by $s(\mathfrak{w})$.

\begin{definition} 
An \emph{admissible} curve is a continuous function $\eta:[0,1]\longrightarrow \Sigma$ such that

1) $\eta(x)\in V$ if and only if  $x\in\{0,1\}$;

2) $\eta$ starts and ends at the common end point of $T_1$ and $T_n$. More precisely, there exists $\epsilon>0$ such that $\eta([0,\epsilon])\subset L_1$ and $\eta([1-\epsilon,1])\subset L_2$;

3) if $\eta(x)\in \mathcal{T}\setminus V$ then $\eta([x-\epsilon,x+\epsilon])$ meets $\mathcal{T}$ transversally for sufficiently small $\epsilon>0$.
\end{definition}

If $\eta$ is admissible, then we obtain $\upsilon(\eta):={i_1}\cdots {i_k}\in \mathfrak W$  given by 
$$\{x\in(0,1) \ : \ \eta(x)\in \mathcal{T}\}=\{x_1<\cdots<x_k\}\quad \text{ and }\quad \eta(x_\ell)\in T_{i_\ell}\text{ for }\ell\in\{1,...,k\}.$$ 
Conversely, note that for every $\mathfrak w\in \mathfrak W$ of an odd length, there is an admissible curve $\eta$ with $\upsilon(\eta)=\mathfrak{w}$. Hence, every reflection in $W$   can be represented by some admissible curve(s). For brevity, let $s(\eta):=s(\upsilon(\eta))$.

\begin{definition} \label{def-rr}
An element $ w \in W$ is called a \emph{rigid reflection}\footnote{It was pointed out by the referee that there is an interesting relationship of this definition with decompositions of the Coxeter element into products of reflections as one can see from \cite{BDSW, IS, Ngu}; in particular, see the proof of Theorem 3.1 in \cite{Ngu}.} if there exist an expression  $w=s_{i_1}s_{i_2} \cdots s_{i_k}$ and a non-self-crossing admissible curve $\eta$ such that  $\upsilon(\eta)=i_1...i_k\in \mathfrak R$. 
\end{definition}

\begin{example} \label{exa-1}
Let $n=4$, and $W = \langle s_1, s_2, s_3, s_4 \ : \ s_1^2=s_2^2=s_3^2=s_4^2 =e \rangle $, i.e., $m_{ij} = \infty$ for $i \neq j$. Consider the curve $\eta$ in the following picture:
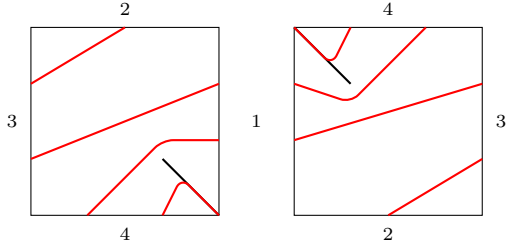
\begin{center}
\begin{tikzpicture}[scale=0.5]
\node at (1,2.5){\tiny{$1$}};
\node at (-2.5,5.5){\tiny{$2$}};
\node at (-5.5,2.5){\tiny{$3$}};
\node at (-2.5,-0.5){\tiny{$4$}};
\node at (4.5,5.5){\tiny{$4$}};
\node at (7.5,2.5){\tiny{$3$}};
\node at (4.5,-0.5){\tiny{$2$}};
\draw (0,0)  -- (0,5) -- (-5,5) -- (-5,0) -- cycle;
\draw (2,0)  -- (2,5) -- (7,5) -- (7,0) -- cycle;
\draw [thick] (0,0) -- (-1.5,1.5); 
\draw [thick] (2,5) -- (3.5,3.5); 
\draw [red, thick, rounded corners] (0,0) -- (-1.0,1.0) --(-1.5,0) ; 
\draw [red, thick, rounded corners] (5.5,5) -- (3.5, 3) -- (2, 3.5); 
\draw [red, thick, rounded corners] (0, 3.5) -- (-5, 1.5); 
\draw [red, thick, rounded corners] (7, 1.5) -- (4.5, 0); 
\draw [red, thick, rounded corners] (-2.5, 5) -- (-5, 3.5); 
\draw [red, thick, rounded corners] (7, 3.5) -- (2, 2); 
\draw [red, thick, rounded corners] (0, 2) -- (-1.5, 2) -- (-3.5, 0); 
\draw [red, thick, rounded corners] (3.5, 5) -- (3, 4)-- (2,5);
\end{tikzpicture}
\end{center}
Since there is no self-intersection, we obtain the corresponding rigid reflection  
\[ s(\eta) = s_4s_1s_3s_2s_3s_1s_4. \]
On the other hand, the reflection $s_4s_2s_3s_1s_3s_2s_4$ comes from the following curve $\eta'$ which has a self-intersection. 
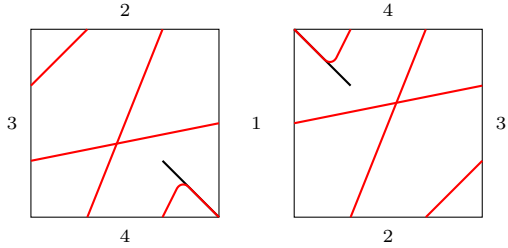
\begin{center}
\begin{tikzpicture}[scale=0.5]
\node at (1,2.5){\tiny{$1$}};
\node at (-2.5,5.5){\tiny{$2$}};
\node at (-5.5,2.5){\tiny{$3$}};
\node at (-2.5,-0.5){\tiny{$4$}};
\node at (4.5,5.5){\tiny{$4$}};
\node at (7.5,2.5){\tiny{$3$}};
\node at (4.5,-0.5){\tiny{$2$}};
\draw (0,0)  -- (0,5) -- (-5,5) -- (-5,0) -- cycle;
\draw (2,0)  -- (2,5) -- (7,5) -- (7,0) -- cycle;
\draw [thick] (0,0) -- (-1.5,1.5); 
\draw [thick] (2,5) -- (3.5,3.5); 
\draw [red, thick, rounded corners] (0,0) -- (-1.0,1.0) --(-1.5,0) ; 
\draw [red, thick, rounded corners] (5.5,5) -- (3.5, 0); 
\draw [red, thick, rounded corners] (-3.5, 5) -- (-5, 3.5); 
\draw [red, thick, rounded corners] (7, 3.5) -- (2, 2.5); 
\draw [red, thick, rounded corners] (0, 2.5) -- (-5, 1.5); 
\draw [red, thick, rounded corners] (7, 1.5) -- (5.5, 0); 
\draw [red, thick, rounded corners] (-1.5, 5) --  (-3.5, 0); 
\draw [red, thick, rounded corners] (3.5, 5) -- (3, 4)-- (2,5);
\end{tikzpicture}
\end{center}
Consequently, the reflection $s(\eta')=s_4s_2s_3s_1s_3s_2s_4$ is {\em not} rigid.

\end{example}

Let $\Phi$ be the root system of $W$, realized in the real vector space $\mathbf E$ with basis $\{\alpha_1, \dots , \alpha_n \}$ with the symmetric bilinear form $B$ defined by
\[ B(\alpha_i, \alpha_j)= - \cos (\pi/m_{ij})  \text{ for } 1 \le i,j \le n.\] For each $i \in \{ 1, \dots , n \}$, define the action of $s_i$ on $\mathbf E$ by 
\[ s_i (\lambda) = \lambda -2B(\lambda, \alpha_i) \alpha_i , \quad \lambda \in \mathbf E, \]
and extend it to the action of $W$ on $\mathbf E$. Then each root $\alpha \in \Phi$ determines a reflection $s_\alpha \in W$.  (See \cite{Hu} for more details.)

\begin{remark}
When $W$ is the Weyl group of a Kac--Moody algebra $\mathfrak g$, the set $\Phi$ is precisely the set of {\em real} roots of $\mathfrak g$. However, $W$ may not be associated with a Kac--Moody algebra, in general. We follow \cite{Hu} to call $\Phi$ the set of roots of $W$.
\end{remark}

\begin{definition} \label{def-rroot}
A positive root $\alpha \in \Phi$ of $W$ is called {\em rigid} if the corresponding reflection $s_\alpha \in W$ is rigid. 
\end{definition}

\begin{example}
In Example \ref{exa-1}, we obtained the rigid reflection $$s_4s_1s_3s_2s_3s_1s_4.$$ It give rises to a rigid root  
\begin{equation*} 6\alpha_1+\alpha_2+2\alpha_3+18\alpha_4=s_4s_1s_3(\alpha_2). \end{equation*}
\end{example}

\medskip

\section{A family of rank 3 Coxeter groups} \label{rank3}

In this section we focus our attention to the rank 3 groups $W(m)$ and collect known results from \cite{LL1} about the rigid reflections of $W(m)$.

\medskip

As in the introduction, fix a positive integer $m\geq 2$ and  set
 $$W(m)=\langle s_1,s_2,s_3 \ : \   s_1^2=s_2^2=s_3^2=(s_1s_2)^m=(s_2s_3)^m=e \rangle.
 $$ Note that we put, in particular, $m_{13}=m_{31}=\infty$.
Let  $\mathcal H(m)$ be the rank 2 hyperbolic Kac--Moody algebra associated with the Cartan matrix ${\tiny \begin{pmatrix}  2&-m\\-m&2\end{pmatrix}}$. We denote an element of the root lattice of $\mathcal H(m)$ by $[a,b]$, $a, b \in \mathbb Z$, where $[1,0]$ and $[0,1]$ are the positive simple roots. 
A root $[a,b]$ of $\mathcal H(m)$ is called {\em reduced} if $\gcd(a,b)=1$ and $ab\neq 0$. 
One can see that every non-simple real root is reduced.

Let $\mathcal P^+=\{ [a,b]\,:\, a, b \in \mathbb Z_{> 0}, \ \gcd(a,b)=1 \}$.
For every $[a,b] \in \mathcal P^+$, let $\eta([a,b])$ be the line segment  from $(0,0)$ to $(a,b)$ on the universal cover of the torus, which automatically has no self-intersections.   Write   
$s([a,b]):=s(\eta([a,b]))\in W(m)$ for the corresponding rigid reflection.
For example, we have \begin{equation} \label{pic-ab}s([5,3])=s_2s_3s_2s_1s_2s_3s_2s_3s_2s_1s_2s_3s_2 \end{equation}  as one can check in the following picture.
\begin{center}\begin{tikzpicture}[scale=0.25mm]
\draw [help lines] (0,0) grid (5,3);
\draw [help lines] (0,1)--(1,0);
\draw [help lines] (0,2)--(2,0);
\draw [help lines] (0,3)--(3,0);
\draw [help lines] (1,3)--(4,0);
\draw [help lines] (2,3)--(5,0);
\draw [help lines] (3,3)--(5,1);
\draw [help lines] (4,3)--(5,2);
\draw [thick,red] (0,0)--(5,3);  
\end{tikzpicture}
\qquad 
\begin{tikzpicture}[scale=0.4mm]
\draw [help lines] (0,1)--(1,0);
\draw [help lines] (0,0)--(1,0);
\draw [help lines] (0,0)--(0,1);
\node at (-0.1, 0.5){\tiny{$3$}};
\node at (0.5, -0.1){\tiny{$1$}};
\node at (0.65, 0.55){\tiny{$2$}};
\end{tikzpicture}
\end{center}

Recall from \cite{Kac, KaMe} that
\begin{equation} \label{eqn-root}
\text{
$[a,b]$ is a root of $\mathcal H(m)$ \quad if and only if \quad $a^2+b^2-mab \le 1$. }
\end{equation}
Define a sequence $\{F_n\}$ recursively by $F_0=0$, $F_1=1$, and $F_n=mF_{n-1}-F_{n-2}$.
Note that $[a,b]$ is a {\em real} root if and only if $[a,b]$ is either $[F_n,F_{n+1}]$ or $[F_{n+1},F_{n}]$, $n \geq 0$.  (See \cite{Kac,KaMe}.) A non-real root is called {\em imaginary}. 

\begin{definition}
Let $(a_1, a_2)$ be a pair of positive integers with $a_1\geq a_2$. 
\begin{enumerate}
\item A \emph{maximal Dyck path} of type $a_1\times a_2$, denoted by $\mathcal{D}^{a_1\times a_2}$, is a lattice path
from $(0, 0)$  to $(a_1,a_2)$ that is as close as possible to the diagonal joining $(0,0)$ and $(a_1,a_2)$ without ever going above it. 

\item Assign $s_2s_3\in W(m)$ to each horizontal edge of $\mathcal{D}^{a_1\times a_2}$, and $s_2s_1\in W(m)$ to each vertical edge. Read these elements in the order of edges along $\mathcal{D}^{a_1\times a_2}$, then we get a product of copies of $s_2s_3$ and $s_2s_1$. Denote the product by $s^{a_1\times a_2}$.
\end{enumerate}
\end{definition}

\begin{remark} \label{rmk-Ch}
It was pointed out by the referee that a maximal Dyck path corresponds to what is known in the literature on combinatorics of words as a {\em Christoffel word}. The statistics associated with a maximal Dyck path defined in the next section of this paper have their counterparts in combinatorics of words. See Chapter 2 of \cite{Lo} for details.
\end{remark}
 
\begin{lem}[\cite{LL1}]  \label{sFnFn1} Assume that $[a,b] \in \mathcal P^+$ with $a\geq b$. Then we have the following formulas.
\begin{enumerate}
\item [(1)] \qquad $s([a,b])=s_3s_2 s^{a\times b} s_1$.
\item [(2)] \qquad  
$s^{F_2\times F_1} =s_2s_1 \quad \text{ and }$
\begin{align*}
s^{F_{n}\times F_{n-1}}&=\left\{\begin{array}{ll}s_1(s_3s_2s_1)^{(n-3)/2}s_2s_3(s_1s_2s_3)^{(n-3)/2}s_1, & \qquad \text{ for }n\geq3\text{ odd};\\
 s_1(s_3s_2s_1)^{(n-4)/2}s_3s_1s_2s_3(s_1s_2s_3)^{(n-4)/2}s_1, & \qquad \text{ for }n\geq4\text{ even.}\end{array} \right .  
\end{align*}

\end{enumerate}
\end{lem} 

We now state one of the main results of this paper.

\begin{thm} \label{thm-main}
The function, $[a,b] \mapsto s^{a \times b}$, is an injection from the set of reduced positive roots $[a,b]$ of $\mathcal H(m)$ with $a \ge b$ into  $W(m)$. 
\end{thm}

As will be shown in Section \ref{sub-in}, Theorem \ref{thm-main} and Corollary \ref{cor-agb} below which establishes a dichotomy between the cases $a>b$ and $a<b$ together imply Theorem \ref{thm-in}. After all necessary constructions and computations are established, a proof of Theorem \ref{thm-main} will be completed in Section \ref{sub-pm}.
\medskip

	\section{Canonical sequences and levels} \label{can sequences}

	In this section we introduce canonical sequences and levels attached to a positive reduced root $[a,b]$ $(a \ge b)$ of $\mathcal H(m)$, which will play an important role in computations of $s^{a \times b}$ and in  classifying reduced roots and rigid reflections.
	Recall that a positive reduced real root $[a,b]$ $(a \ge b)$ of $\mathcal H(m)$ is of the form $[F_n, F_{n-1}]$. Since we already have $s^{F_n \times F_{n-1}}$ in Lemma \ref{sFnFn1}, we only consider imaginary roots of $\mathcal H(m)$ whenever this simplifies the presentation.

	\medskip

	We start with some definitions. (Recall Remark \ref{rmk-Ch}; {\em cf.} \cite{LS,Lo}.)

	\begin{definition}
	Let $N$ be a positive integer. Suppose that $\bc=(a_1,a_2, \dots , a_d)$ is a finite sequence such that $a_i=N$ or $N+1$ for all $i$. 
	\begin{itemize}
	\item[(1)] If $d>1$ and $(a_i,a_{i+1})\neq (N,N)$ for any $i $, then $\bc$ is called  {\it type $+$}.
	\item[(2)]  If $d>1$ and $(a_i,a_{i+1})\neq (N+1,N+1)$ for any $i$, then $\bc$ is called  {\it type $-$}.

	\item[(3)] If $d>1$ and $a_i \neq a_{i+1}$ for any $i$, then $\bc$ is called  {\it type $=$}.
	\item[(4)] If $d=1$, then $\bc$ is called  {\it type $0$}.

	\end{itemize}
	\end{definition}

	Throughout this section, let $[a,b]$ be a reduced positive root of $\mathcal H(m)$ with $a\ge b$. 

	\begin{definition}
	Define a sequence  $\bc_1 =( a_{1,1},a_{1,2}, \dots,a_{1,b} )$ of positive integers  to be such that 
	\begin{equation} \label{eqn-sab}
	\mathcal D^{a \times b}=h^{a_{1,1}}vh^{a_{1,2}}v\cdots h^{a_{1,b}}v,
	\end{equation}
	where $h$ is a horizontal edge and $v$ is a vertical edge and the product means concatenation.
	\end{definition}

	\begin{lem}\label{lem:1stseq}
	We have 
	\begin{equation}\label{eq:a_{1i}}
	a_{1,i}=\left\lceil\frac{ai}{b}\right\rceil-\left\lceil\frac{a(i-1)}{b}\right\rceil \quad (1 \leq i \leq b).
	\end{equation}
	\end{lem}
	\begin{proof} 
	Since the slope of the line $\eta([a,b])$ is $\frac{b}{a}$, the number $a_{1,1}$ is the smallest positive integer such that $\frac{b}{a}a_{1,1} \geq 1$. Thus, we obtain $a_{1,1}=\lceil\frac{a}{b}\rceil$. Since the $y$-coordinate of the point on the line with the $x$-coordinate $a_{1,2}$ is greater than or equal to 2, the number $a_{1,2}$ is the smallest positive integer such that $\frac{b}{a}(a_{1,1}+a_{1,2})\geq2$. Thus, we obtain $a_{1,2}=\lceil\frac{2a}{b}\rceil-\lceil\frac{a}{b}\rceil.$ Now assume that $a_{1,i}=\lceil\frac{ai}{b}\rceil-\lceil\frac{a(i-1)}{b}\rceil$ for $1\leq i \leq k$. By a similar argument, $a_{1,k+1}$ is the smallest positive integer such that $\frac{b}{a}(a_{1,1}+\cdots+a_{1,k+1})\geq k+1$. Since $(a_{1,1}+\cdots+a_{1,k})=\lceil\frac{ak}{b}\rceil$, we obtain $a_{1,k+1}=\lceil\frac{a(k+1)}{b}\rceil-\lceil\frac{ak}{b}\rceil$. By induction, we are done.
	\end{proof}
	 
	\begin{lem}\label{lem:seqinj}
	The function $[a,b] \mapsto  ( a_{1,1},\dots,a_{1,b} )$ is an injection from the set of reduced positive roots into the set of finite sequences in $\mathbb{Z}_{>0}$. 
	\end{lem}

	\begin{proof} It follows directly from \eqref{eq:a_{1i}}.
	\end{proof}

	Recall the assumption that $[a,b]$ is a reduced positive root of $\mathcal H(m)$ with $a\ge b$. 

	\begin{lem}\label{lem:first}
	Assume that $\frac a b \neq m$. Let 
	\begin{equation} \label{n1r1}
	\frac{a}{b}=N_1+\rho_1 \quad \text{with $ N_1=\left\lfloor\frac{a}{b}\right\rfloor.$}
	\end{equation}
	 Then we have
	\begin{itemize}
	\item[(1)] $1\leq N_1 \leq m-1$;
	\item[(2)]$a_{1,i}=N_1$ or $N_1+1$ for all $1\leq i \leq b$;
	\item[(3)] $a_{1,1}=N_1+1$ and $a_{1,b}=N_1$;
	\item[(4)] $\bc_1=(a_{1,i})$ is of type $
	\begin{cases}
	+&\text{if $\rho_1 > \frac{1}{2}$,}\\
	-&\text{if $0<\rho_1 < \frac{1}{2}$,}\\
	=&\text{if $\rho_1 = \frac{1}{2}$.}
	\end{cases}
	$
	\end{itemize}
	\end{lem}
	\begin{proof}
	(1) Since $[a,b]$ is a reduced positive root of $\mathcal{H}(m)$, we have $a^2+b^2-mab \leq 1$ from \eqref{eqn-root}. Thus we have $1 \leq  \frac{a}{b} \leq m$. Since we assume that $\frac{a}{b} \neq m$, we obtain the desired result.

	(2), (3) It is clear from the fact that $a_{1,i}=N_1+\lceil \rho_1i\rceil -\lceil \rho_1(i-1) \rceil$.

	(4) Suppose $\rho_1>\frac{1}{2}$.
	We will show that $(a_{1,i},a_{1,i+1})\neq(N_1,N_1)$ for all $i$. Note that  $a_{1,i}=N_1$ if and only if $\lceil \rho_1i\rceil -\lceil \rho_1(i-1) \rceil=0$.
	If $(a_{1,i},a_{1,i+1})=(N_1,N_1)$ for some $i$, then
	\begin{equation*}
	\lceil \rho_1(i+1)\rceil -\lceil \rho_1i \rceil=\lceil \rho_1i\rceil -\lceil \rho_1(i-1) \rceil=0 \
	\Longleftrightarrow \ \lceil \rho_1(i+1)\rceil=\lceil \rho_1i\rceil=\lceil \rho_1(i-1)\rceil
	\end{equation*}
	which implies that there exists an integer $t$ such that $t < \rho_1(i-1)<\rho_1i<\rho_1(i+1)< t+1$. This contradicts to $\rho_1 > \frac{1}{2}$.
	The other cases can be proved similarly.
	\end{proof}

	\begin{definition}\label{def:a_n}
	Let $[a,b]$ be a reduced positive root of $\mathcal H(m)$ with $a\ge b$. For $n \ge 1$, define inductively $\rho_n, N_n$ and $\bc_n= (a_{n,1},a_{n,2},\cdots,a_{n,d_n})$ as follows.
	 \begin{itemize}
	\item[(0)] Note that $\rho_1, N_1$ and $\bc_1=(a_{1,1}, \dots , a_{1,d_1})$ with $d_1=b$ are already defined in \eqref{eqn-sab} and \eqref{n1r1}. 

	\item[(1)] If $\bc_{n-1}$ is of type $=$ or $0$, stop the process. Otherwise, $\bc_n=( a_{n,i} )_{1\leq i \leq d_n}$ are defined to be the sequence  recording the numbers of consecutive occurrences of
	\begin{equation*}
	\begin{cases}
	 \text{ $(N_{n-1}+1)$'s in $\bc_{n-1}$} & \text{if $\bc_{n-1}$ is of type $+$,}\\
	 \text{ $N_{n-1}$'s in $\bc_{n-1}$} & \text{if $\bc_{n-1}$ is of type $-$,}\\
	\end{cases}
	\end{equation*}
	 where $d_n$ is the number of $N_{n-1}$ (resp. $N_{n-1}+1$) in $\bc_{n-1}$ if it is of type $+$ (resp. type $-$).  
	\item[(2)] $\rho_n$ is defined to be a rational number with $0 \leq \rho_n <1$ and $N_n$ is to be a positive integer such that
	 \begin{equation*}
	N_{n}+\rho_{n}=
	\begin{cases}
	\frac{\rho_{n-1}}{1-\rho_{n-1}} &\text{if $\rho_{n-1} \ge \frac{1}{2}$}, \vspace*{0.2 cm}\\
	\frac{1-\rho_{n-1}}{\rho_{n-1}} &\text{if $\rho_{n-1} < \frac{1}{2}$}.
	\end{cases}
	\end{equation*}
	 \end{itemize}
	The sequences $\bc_n$, $n=1, 2, \dots $, are called the {\em canonical sequences} of $[a,b]$.
	 \end{definition}

	\begin{example}\label{ex:[5,3]}
	(1) Let $m=3$ and $[a,b]=[5,3]$. Then $N_1=1$, $\rho_1=\frac{2}{3}$ and $d_1=b=3$.
	From the definition or  by Lemma \ref{lem:1stseq}, the sequence $\bc_1$ is given by 
	 \begin{equation*}
	 \bc_1=( a_{1,1},a_{1,2},a_{1,3} )=( 2,2,1 ),
	 \end{equation*}
	 which is of type $+$.  Since
	 \begin{equation*}
	 \frac{\rho_1}{1-\rho_1}=\frac{\frac{2}{3}}{1-\frac{2}{3}}=2=N_2,
	 \end{equation*}
	 we have $\bc_2=(a_{2,1})=(2)$, which is of type $0$.

	(2) Let $m=3$ and $[a,b]=[8,5]$. In this case, $N_1=1, \rho_1=\frac{3}{5}, d_1=b=5$ and 
	\begin{equation*}
	\bc_1=(a_{1,1},a_{1,2},a_{1,3},a_{1,4},a_{1,5})=(2,2,1,2,1),
	\end{equation*}
	 which is of type $+$.
	 Then $N_2=1, \rho_2=\frac{1}{2},  d_2=2$ and 
	 \begin{equation*}
	\bc_2=(a_{2,1},a_{2,2})=(2,1),
	\end{equation*}
	 which is of type $=$.
	 
	(3) Assume $m=3$ and $[a,b]=[59,23]$. Then we have
	\begin{align*}
	\bc_1 &=(
	3, 3, 2, 3, 2, 3, 2, 3, 3, 2, 3, 2, 3, 2, 3, 3, 2, 3, 2, 3, 2, 3, 2),& N_1&=2, & \rho_1&=13/23, &\text{type}& \ +, \\ \bc_2 &= (
	2, 1, 1, 2, 1, 1, 2, 1, 1, 1),&N_2&=1,&\rho_2&=3/10,&\text{type}& \ - , \\ \bc_3 &=(
	2, 2, 3), &N_3&=2,&\rho_3&=1/3,&\text{type}& \ - , \\ \bc_4&=(
	2),&N_4&=2,&\rho_4&=0,&\text{type}& \ 0 . 
	\end{align*}

	(4) Suppose $m=5$ and $[a,b]=[62,13]$. Then we obtain
	\begin{align*}\bc_1 &=(5, 5, 5, 5, 4, 5, 5, 5, 4, 5, 5, 5, 4),  &N_1&=4,&\rho_1&=10/13,&\text{type}& \ + , \\  \bc_2 &= (4, 3, 3), &N_2&=3,&\rho_2&=1/3,&\text{type}& \ - ,  \\ \bc_3&=  (2), &N_3&=2,&\rho_3&=0,&\text{type}& \ 0 .
	\end{align*}
	\end{example}

	For a positive rational number $r$, let $D(r)=q$ when $r=\frac{p}{q}$ and $p,q$ are relatively prime integers.   
	\begin{lem}\label{lem:a_{n,i}}
	The following holds for $n \ge 2$. 
	 \begin{itemize}
	 \item[(1)]  $\bc_n=(a_{n,i})_{1 \le i \le d_n}$ is given by 
	\begin{equation*}
	a_{n,i}=
	\begin{cases}
	\left\lceil\frac{\rho_{n-1} i}{1-\rho_{n-1}}\right\rceil-\left\lceil\frac{\rho_{n-1}(i-1)}{1-\rho_{n-1}}\right\rceil, &\text{when $\bc_{n-1}$  is of type $+$}, \vspace*{0.3 cm} \\ 
	\left\lfloor\frac{(1-\rho_{n-1})i}{\rho_{n-1}}\right\rfloor-\left\lfloor\frac{(1-\rho_{n-1})(i-1)}{\rho_{n-1}}\right\rfloor, & \text{when $\bc_{n-1}$ is of type $-$},\\
	\end{cases}
	\end{equation*}
	for $1 \leq i \leq d_n$, and we get $d_n=D(\rho_n)$.
	\item[(2)]We have $a_{n,i}=N_n$ or $N_n+1$ for $1\leq i \leq d_n$.
	\item[(3)]  $a_{n,1}=N_n+1$ {\rm (}resp. $N_n${\rm )} and $a_{1,d_n}=N_n$ {\rm (}resp. $N_n+1$ {\rm )} if  $\rho_n\neq0$ and $\bc_{n-1}$ is of type $+$ {\rm (}resp. type $-${\rm )}.
	\item[(4)] $\bc_n=(a_{n,i})$ is of type $
	\begin{cases}
	+&\text{if $\rho_n > \frac{1}{2}$,}\\
	-&\text{if $0<\rho_n < \frac{1}{2}$,}\\
	= &\text{if $\rho_n = \frac{1}{2}$,}\\
	0 &\text{if $\rho_n = 0$.}
	\end{cases}
	$
	\end{itemize}
	 \end{lem}
	 \begin{proof}
	 (1)
	 Suppose that $\bc_{n-1}$ is of type $+$.
	We have 
	 \begin{equation*}
	 \begin{split}
	 a_{n-1,i}&=\lceil (N_{n-1}+\rho_{n-1})i \rceil-\lceil (N_{n-1}+\rho_{n-1})(i-1) \rceil\\
	 &=N_{n-1}+\lceil\rho_{n-1}i\rceil-\lceil \rho_{n-1}(i-1) \rceil.
	 \end{split}
	 \end{equation*}
	 Since $a_{n,1}$ is the number of first successive  $N_{n-1}+1$ in the sequence $\bc_{n-1}$, we have
	\begin{equation*}
	\begin{split}
	\lceil \rho_{n-1}i \rceil-\lceil \rho_{n-1}(i-1) \rceil=
	\begin{cases}
	1 &\text{if $i=1,2,\cdots,a_{n,1}$},\\
	0 &\text{if $i=a_{n,1}+1$}.
	\end{cases}
	\end{split}
	\end{equation*}
	This implies that 
	\begin{equation*}
	  (a_{n,1}-1)\rho_{n-1} <a_{n,1}-1<a_{n,1}\rho_{n-1}<(a_{n,1}+1)\rho_{n-1} \leq a_{n,1}.
	\end{equation*}
	  So we obtain 
	\begin{equation*}
	  a_{n,1}\geq 1, \quad\frac{\rho_{n-1}}{1-\rho_{n-1}} \leq a_{n,1} < \frac{1}{1-\rho_{n-1}},
	\end{equation*}
	  and hence $a_{n,1}=\lceil \frac{\rho_{n-1}}{1-\rho_{n-1}}\rceil$. 
	  
	  Since $a_{n,2}$ is the number of successive $N_{n-1} +1$  in $\bc_{n-1}$ between the first and the second $N_{n-1}$, we have
	  \begin{equation*}
	  \begin{split}
	   \lceil \rho_{n-1}i \rceil-\lceil \rho_{n-1}(i-1) \rceil=
	   \begin{cases}
	 1 &\text{ if $i=a_{n,1}+2,a_{n,1}+3,\cdots,a_{n,1}+a_{n,2}+1$}, \\
	 0 &\text{ if $i=a_{n,1}+a_{n,2}+2$}.
	 \end{cases}
	  \end{split}
	  \end{equation*}
	  This implies that
	  \begin{equation*}
	  (a_{n,1}+a_{n,2})\rho_{n-1} < a_{n,1}+a_{n,2}-1 < (a_{n,1}+a_{n,2}+1)\rho_{n-1} < (a_{n,1}+a_{n,2}+2)\rho_{n-1} \leq a_{n,1}+a_{n,2}.
	  \end{equation*}
	  Hence, we have 
	  \begin{equation*}
	   \frac{2\rho_{n-1}}{1-\rho_{n-1}}-a_{n,1} \leq a_{n,2} < -a_{n,1}+\frac{\rho_{n-1}+1}{1-\rho_{n-1}} .
	   \end{equation*}
	   Since $a_{n,2}$ is an integer and $\frac{\rho_{n-1}+1}{1-\rho_{n-1}}- \frac{2\rho_{n-1}}{1-\rho_{n-1}}=1$, we obtain
	   \begin{equation*}
	   a_{n,2}=\left\lceil\frac{2\rho_{n-1}}{1-\rho_{n-1}}\right\rceil-\left\lceil\frac{\rho_{n-1}}{1-\rho_{n-1}}\right\rceil.
	   \end{equation*}

	By a similar argument, we have 
	  \begin{equation*}
	  a_{n,i}=\left\lceil\frac{\rho_{n-1}i}{1-\rho_{n-1}}\right\rceil-\left\lceil\frac{\rho_{n-1}(i-1)}{1-\rho_{n-1}}\right\rceil \qquad \text{ for } i \geq 3 .
	  \end{equation*}
	  
	  Next, we show that $d_n=D(\rho_n)$.
	  By the definition of $\bc_n=(a_{n,i})_{1 \le i \le d_n }$, we have
	  \begin{equation*}
	  d_{n-1}=\sum_{1\leq i \leq d_n}a_{n,i}+d_n=\left\lceil \frac{\rho_{n-1}d_n}{1-\rho_{n-1}}\right\rceil +d_n.
	  \end{equation*}
	  Suppose that $\rho_{n-1}=\frac{e_{n-1}}{d_{n-1}}$, where $e_{n-1}$ and $d_{n-1}$ are relatively prime. Then we have
	  \begin{equation*}
	 N_n+\rho_n= \frac{\rho_{n-1}}{1-\rho_{n-1}}=\frac{e_{n-1}}{d_{n-1}-e_{n-1}}.
	  \end{equation*}
	  Since $d_{n-1}$ and $e_{n-1}$ are relatively prime, $D(\rho_n)=D(\frac{e_{n-1}}{d_{n-1}-e_{n-1}})=d_{n-1}-e_{n-1}$.
	 Hence, we obtain
	 \begin{equation*}
	 d_{n-1}=\left\lceil \frac{e_{n-1}d_n}{d_{n-1}-e_{n-1}}\right\rceil +d_n \ \
	\Longleftrightarrow \ \ d_n=d_{n-1}-e_{n-1}=D(\rho_n).
	 \end{equation*}
	  
	 The proof for the case when $\bc_{n-1}$ is of type $-$ is similar, and we omit the details.
	 
	  (2) Suppose that $\bc_{n-1}$ is of type $+$.
	  Since $\frac{\rho_{n-1}}{1-\rho_{n-1}}=N_n+\rho_n$,  we have
	\begin{equation*}
	a_{n,i}=\lceil N_ni+\rho_ni \rceil -\lceil N_n(i-1)+\rho_n(i-1) \rceil=N_n+\lceil \rho_ni\rceil -\lceil \rho_n(i-1)\rceil.
	\end{equation*}
	Since $0 \leq \rho_n <1$, we obtain $a_{n,i}=N_n$ or $N_n+1$ for $1\leq i \leq d_n$.
	The case $\bc_{n-1}$ is of type $-$ is similar.

	(3) Suppose that $\rho_n \neq 0$ and $\bc_{n-1}$ is of type $+$.
	Write $\rho_{n-1}= \frac {e_{n-1}}{d_{n-1}}$ with $e_{n-1}$ and $d_{n-1}$ relatively prime. Then $N_n+\rho_n = \frac {e_{n-1}}{d_{n-1} - e_{n-1}}$. Since $\rho_n \neq 0$, we have
	\begin{equation*}
	\begin{split}
	a_{n,1}=&N_n+\left\lceil \rho_n\right\rceil= N_n+1, \\
	a_{n,d_n}=&\left\lceil (N_n+ \rho_n) d_n \right\rceil -\left\lceil(N_n+ \rho_n)(d_n-1)\right\rceil\\
	=&\left\lceil \frac{e_{n-1}}{d_{n-1}-e_{n-1}}d_n\right\rceil -\left\lceil \frac{e_{n-1}}{d_{n-1}-e_{n-1}}(d_n-1)\right\rceil\\
	=&\left\lceil e_{n-1} \right\rceil - \left\lceil e_{n-1}-N_n-\rho_n \right\rceil 
	=e_{n-1}-e_{n-1}+N_n=N_n.
	  \end{split}
	  \end{equation*}
	The case $\bc_{n-1}$ is of type $-$ is similar.

	(4) The proof is similar to that of Lemma \ref{lem:first} (4), and we omit the details.
	\end{proof}

\begin{lem} \label{lem:ncninj} Let $[a,b]$ be a reduced positive root of $\mathcal H(m)$ with $a \ge b$, and $N_k$ and $\bc_k$ be defined as in Definition \ref{def:a_n}. Denote the type of $\bc_k$ by $\epsilon_k$. For each $k \ge 0$, the data
\[ (N_1,\epsilon_1, N_2, \epsilon_2, \dots , N_k, \epsilon_k, \bc_{k+1})\] determines $[a,b]$ uniquely.
\end{lem}

\begin{proof}
By Lemma  \ref{lem:a_{n,i}} and Definition \ref{def:a_n}, we obtain $\bc_{k}$ from $(N_k, \epsilon_k, \bc_{k+1})$ and continue the process to obtain $\bc_1$. Now the assertion follows from 
Lemma \ref{lem:seqinj}.
\end{proof}

	After establishing another lemma below, we will define the level of $[a,b]$.
	 
	\begin{lem}\label{lem:range}
	Assume that $[a,b]$ is an imaginary reduced positive root of $\mathcal H(m)$ with $a\ge b$. For $n \ge 2$, if $N_k=m-2+\delta_{1,k}$ and $\bc_k$ is of type $+$ for $1\leq k \leq n-1$, then $$1< N_n+\rho_n < \gamma-1,$$ where we set 
	$ \gamma :=\frac{m+\sqrt{m^2-4}}{2} .$ 
	In particular, $1\leq N_{n} \leq m-2$ for $n \ge 2$. 
	\end{lem}
	\begin{proof}
	We use induction on $n$. It follows from the assumptions and Lemma \ref{lem:first} that $\frac{a}{b}=N_1+\rho_1=m-1+\rho_1$ with $\frac{1}{2} < \rho_1 <1$ and
	$N_2+\rho_2=\frac{\rho_1}{1-\rho_1}$.
	Since $[a,b]$ is a reduced positive root of $\mathcal H(m)$ and $[a,b] \neq [F_i,F_{i-1}]$ for any $i =2,3,\dots$, we have $1 \leq \frac{a}{b} <\gamma$ and 
	$\frac{1}{2}< \rho_1 <\gamma-(m-1)$. Note that $\gamma ( m-\gamma)=1$. 
	Since $y=\frac{x}{1-x}$ is an increasing  function for $0 \leq x < 1$, we obtain
	\begin{equation*}
	\frac{\frac{1}{2}}{1-\frac{1}{2}}=1 <N_2+\rho_2=\frac{\rho_1}{1-\rho_1} <\frac{\gamma-m+1}{-\gamma+m}=\gamma-1 <m-1.
	\end{equation*}
	Hence, $1\leq N_2 \leq m-2$.

	Now assume that we have $1< N_{n-1}+\rho_{n-1} < \gamma-1$.
	Since $N_{n-1}=m-2$ and $\bc_{n-1}$ is of type $+$, we have $\frac{1}{2} < \rho_{n-1} < \gamma-(m-1)$.
	By the same argument as in the case $n=2$, we have
	\begin{equation*}
	1< N_n+\rho_n < \gamma-1 \quad \text{ and } \quad 1 \le N_n \le m-2.
	\end{equation*}
	\end{proof}

	\medskip

	 Let us consider the sequence $\{\gamma_n\}$ given by
	\begin{equation*}
	\begin{split}
	&\gamma_0=0, \qquad \gamma_1=m-\frac{1}{2},\qquad
	\gamma_n=m-\frac 1 {\gamma_{n-1}}\quad (n\geq 2).
	 \end{split}
	 \end{equation*}
	 It is straightforward to check that $\gamma_n <\gamma_{n+1} <\gamma$ for $n \geq 1$, and
	 \begin{equation*}
	\gamma_n \rightarrow \gamma  \qquad \text{ as }\ n \rightarrow \infty,
	\end{equation*}
	where $\gamma =  \frac{m+\sqrt{m^2-4}}{2}$ as before.
	 
	\begin{definition}\label{def:level}
	Let $[a,b]$ be an imaginary reduced positive root of $\mathcal H(m)$ with $a \ge b$.
	Then the {\em level} $L$ of $[a,b]$ is defined to be the positive integer uniquely determined by the inequalities
	\begin{equation*}
	\gamma_{L-1} < \frac{a}{b}\leq \gamma_{L}.
	\end{equation*} 
	\end{definition}

	 \begin{example}
	 Let $m=3$ and $[a,b]=[339,130]$. Then we have $\gamma_1=\frac{5}{2}, \gamma_2=\frac{13}{5}, \gamma_3=\frac{34}{13}$ and $\gamma_2<\frac{a}{b}<\gamma_3$. Thus the level of $[a,b]$ is 3. 
	 \end{example}

	\begin{prop} \label{prop-Nk}
	Let $[a,b]$ be an imaginary reduced positive root of $\mathcal H(m)$ with $a\ge b$. 
	Assume that the level of $[a,b]$ is $L$. Then  the following statements hold:
	\begin{itemize}
	\item[(i)] $N_k=m-2+\delta_{1,k}$ for $1\leq k \leq L-1$, and hence 
	\[ (N_1, N_2, \dots, N_L) = (m-1, m-2, m-2, \dots , m-2, c) \] for $1 \le c \le m-2+\delta_{1,L}$, where $\delta_{i,j}$ is the Kronecker's delta.
	 
	\item[(ii)] $\bc_k$ is of type $+$ for $1\leq k \leq L-1$,\item[(iii)] $N_{L}+\rho_{L} \leq m-2+\delta_{1,L}+\frac{1}{2}$.
	\end{itemize}
	\end{prop}

	\begin{proof}
	Suppose the level of $[a,b]$ is $1$. Then we obtain $N_1+\rho_1=\frac{a}{b}\leq \gamma_1=m-\frac{1}{2}$. 
	Suppose the level of $[a,b]$ is $2$. Then $(m-1)+\frac{1}{2}< N_1+\rho_1\leq \gamma_2= (m-1)+\frac{2m-3}{2m-1}$.
	Thus we obtain $\frac{1}{2} < \rho_1 \leq \frac{2m-3}{2m-1}$. It implies that $1< N_2+\rho_2=\frac{\rho_1}{1-\rho_1} \leq  (m-2)+\frac{1}{2}$.
	Hence, $N_1=m-1$, $\bc_1$ is of type $+$ and $N_L+\rho_L \leq m-2+\frac{1}{2}$.

	Now suppose the level of $[a,b]$ is $L\geq 3$.
	 Then $$(m-1)+1-\frac{1}{\gamma_{L-2}}=\gamma_{L-1} <\frac{a}{b} \leq \gamma_L=(m-1)+1-\frac{1}{\gamma_{L-1}}.$$   Thus we obtain $1-\frac{1}{\gamma_{L-2}} <\rho_1 \leq 1-\frac{1}{\gamma_{L-1}}$.
	It implies that $\gamma_{L-2}-1<N_2+\rho_2=\frac{\rho_1}{1-\rho_1}\leq \gamma_{L-1}-1$. Since $m-2<\gamma_{L-2}-1$, we have $N_2=m-2$ and $\gamma_{L-2}-m+2<\rho_2 \leq \gamma_{L-1}-1-m+2$.

	By the recursive relation of $\{\gamma_n\}$, we obtain
	\begin{equation*}
	\gamma_{L-3}-1<N_3+\rho_3=\frac{\rho_2}{1-\rho_2}\leq \gamma_{L-2}-1.
	\end{equation*}
	Repeating this argument yields 
	\begin{equation*}
	m-2+\frac{1}{2}=\gamma_{1}-1<N_{L-1}+\rho_{L-1}=\frac{\rho_{L-2}}{1-\rho_{L-2}}\leq \gamma_{2}-1.
	\end{equation*}
	Hence, $1<N_{L}+\rho_{L}=\frac{\rho_{L-1}}{1-\rho_{L-1}} \leq \gamma_1-1=m-2+\frac{1}{2}$.
	It implies that $N_k=m-2+\delta_{1,k}$ for $1\leq k \leq L-1$.
	Moreover, $\bc_k$ is of type $+$ for $1\leq k \leq L-1$ because $\rho_k>\frac{1}{2}$.
	Since $N_L+\rho_L \leq m-2+\frac{1}{2}$, we obtain the desired result.
	\end{proof}

	\begin{cor} \label{cor-cannot}
        Assume that $L$ is the level of $[a,b]$.
        \begin{enumerate}
            \item If $N_L=m-2+\delta_{1,L}$ then $\bc_L$ cannot be of type $+$.
	\item If $L \ge 2$ and $N_L=1$ then $\bc_L$ cannot be of type $0$.
	\end{enumerate}
	\end{cor}

	\begin{proof}
	Part (1) is an immediate consequence of Proposition \ref{prop-Nk} (iii) and Lemma \ref{lem:a_{n,i}} (4). For part (2), assume that $N_L=1$ and $\bc_L$ of type $0$. Then $\rho_L=0$ by Lemma \ref{lem:a_{n,i}} (4) and $\rho_{L-1}=\frac 1 2$ by the definition of $\rho_n$. However, $\rho_{L-1}> \frac 1 2$ by Proposition \ref{prop-Nk} (ii) and Lemma \ref{lem:a_{n,i}} (4), which is a contradiction.
	\end{proof}

	\subsection{Reduction according to canonical sequences}
	In this subsection we show that $s^{a \times b}$ can be written through the canonical sequences $\bc_k$. This result will be used in the next section and has its own interest.  
	In what follows, we write only the subscripts of simple reflections when we express elements in $W(m)$.  For example, we write $23=s_2s_3$ and $21=s_2s_1$.

	\begin{prop}\label{lem:notreduced}
	Let $[a,b]$ be an imaginary positive reduced root of $\mathcal H(m)$ with $a \ge b$, and $\bc_k=(a_{k,1}, a_{k,2}, \dots , a_{k, d_k})$ its canonical sequences. Suppose the level of $[a,b]$ is $L$. According to the values of $k$, define $H_k, V_k \in W(m)$ by
	\begin{center}
	\begin{tabular}{|c||c|c|}
	\hline
	$k$          & $H_k$                       & $V_k$  \\ \hline \hline
	$1$  & $23$    &   $21$            \\ \hline
	$2$  & $21$    &   $31$            \\ \hline
	$2l+1\ (l\geq1)$  & $1(321)^{l-1}23(123)^{l-1}1$     & $1(321)^{l-1}2123(123)^{l-1}1$               \\ \hline
	$2l+2\ (l\geq1)$  & $13(213)^{l-1}12(312)^{l-1}31$     & $13(213)^{l-1}2312(312)^{l-1}31$               \\ \hline
	\end{tabular}
	\end{center}
	Then, for $1\leq k\leq L$, we have
	\begin{equation} \label{sabH}
	s^{a\times b}= H_k^{a_{{ k},1}}V_k H_k^{a_{{ k},2}}V_k \cdots H_k^{a_{{{k}, d_k}}} V_k.
	\end{equation}
	\end{prop}

\begin{example}
Continuing Example \ref{ex:[5,3]} (4), suppose $m=5$ and $[a,b]=[62,13]$. Since $\gamma_1=\frac 9 2 < \frac {62}{13} < \gamma_2 =\frac {43} 9$, the level of $[62,13]$ is $2$. Since $\bc_2=(4,3,3)$, we obtain \[ s^{62 \times 13}=(21)^4(31)(21)^3(31)(21)^3(31) \in W(5). \]
\end{example}

    \begin{proof}[Proof of Proposition \ref{lem:notreduced}]
	If $k=1$, then \eqref{sabH} follows from the definition of $s^{a \times b}$ and \eqref{eqn-sab}. Suppose $k=2$. By Proposition \ref{prop-Nk}, $N_1=m-1$ and $\bc_1$ is of type $+$. Then, by  Lemma \ref{lem:first} and the definition of $\bc_2$, we have
	\[ \bc_1=(a_{1,1}, a_{1,2}, \dots , a_{1,d_1}) =(m^{a_{2,1}}, m-1, m^{a_{2,2}}, m-1, \dots , m^{a_{2,d_2}}, m-1) ,\] where we write
	$m^s= \underbrace{m,m,\ldots,m}_{s\text{ times}}$. Since $(23)^m=e$ and $(23)^{m-1}21=31$, we obtain
	\begin{equation*}
	\begin{split}
	s^{a\times b}= &H_1^{a_{{ 1},1}}V_1 H_1^{a_{{ 1},2}}V_1 \cdots H_1^{a_{{{1}, d_1}}} V_1\\=& (21)^{a_{2,1}}(31) (21)^{a_{2,2}}(31) \cdots (21)^{a_{2, d_2}}(31)
	= H_2^{a_{{ 2},1}}V_2 H_2^{a_{{ 2},2}}V_2 \cdots H_2^{a_{{{2}, d_2}}} V_2.
	\end{split}
	\end{equation*}

	Suppose that $k=3$. By Proposition \ref{prop-Nk}, $N_2=m-2$ and $\bc_2$ is of type $+$. By Lemma \ref{lem:a_{n,i}}, we have $a_{2,i}=m-2$ or $m-1$ for $1 \leq i \leq d_2$, $a_{2,1}=m-1$ and $a_{2,d_2}=m-2$. 
	Then, by the definition of $\bc_3$, we have
	\begin{align*} \bc_2&=(a_{2,1}, a_{2,2}, \dots , a_{2,d_2}) \\ &=((m-1)^{a_{3,1}}, m-2, (m-1)^{a_{3,2}}, m-2, \dots , (m-1)^{a_{3,d_3}}, m-2) ,\end{align*} where we write
	$(m-1)^s= \underbrace{m-1,m-1,\ldots,m-1}_{s\text{ times}}$. 
	Since $(21)^{m-1}(31)=1231$ and $(21)^{m-2}(31)=121231$, we have
	\begin{equation*}
	\begin{split}
	s^{a\times b}=& H_2^{a_{{ 2},1}}V_2 H_2^{a_{{ 2},2}}V_2 \cdots H_2^{a_{{{2}, d_2}}} V_2\\=&(1231)^{a_{3,1}}(121231) (1231)^{a_{3,2}}(121231) \cdots (1231)^{a_{3, d_3}}(121231)\\
	=& H_3^{a_{{ 3},1}}V_3 H_3^{a_{{ 3},2}}V_3 \cdots H_3^{a_{{{3}, d_3}}} V_3.
	\end{split}
	\end{equation*}

	The proof for $k=4$ is similar to the case $k=3$ with 
	\begin{equation*}
	\begin{split}
	(1231)^{m-1}(121231)&=1(3123)1=13(12)31,\\
	(1231)^{m-2}(121231)&=1(323123)1=13(2312)31.
	\end{split}
	\end{equation*}

Now let us use induction on $k$. Suppose \eqref{sabH} is true for $k=2l+2$ and consider $k+1=2l+3 \le L$. Then $N_{k}=m-2$ and $\bc_{k}$ is of type $+$.  By Lemma \ref{lem:a_{n,i}} and the definition of $\bc_{k+1}$, we have $a_{k,1}=m-1$, $a_{k,d_k}=m-2$ and
\begin{align} \label{bckm-1}
 \bc_{k}&=(a_{k,1}, a_{k,2}, \dots , a_{k,d_k}) \\ &=((m-1)^{a_{k+1,1}}, m-2, (m-1)^{a_{k+1,2}}, m-2, \dots , (m-1)^{a_{k+1,d_{k+1}}}, m-2) . \nonumber
\end{align}
Since we have 
\begin{align*}
H_k^{m-1}V_k &= 13(213)^{l-1} (12)^{m-1} 2312 (312)^{l-1}31 = 1 (321)^l 23 (123)^l 1=H_{k+1},\\
H_k^{m-2}V_k & = 13(213)^{l-1} (12)^{m-2} 2312 (312)^{l-1}31 = 1 (321)^l 2123 (123)^l 1=V_{k+1},
\end{align*}
it follows from \eqref{bckm-1} and the induction hypothesis that 
\begin{equation*}
\begin{split}
s^{a\times b}=& H_{k}^{a_{{ k},1}}V_k H_k^{a_{{ k},2}}V_k \cdots H_k^{a_{{{k}, d_k}}} V_k\\=& H_{k+1}^{a_{{ k+1},1}}V_{k+1} H_{k+1}^{a_{{ k+1},2}}V_{k+1} \cdots H_{k+1}^{a_{{{k+1}, d_{k+1}}}} V_{k+1}.
\end{split}
\end{equation*}

The proof for the next step (i.e. $k=2l+3$) is similar, and we omit the details.
\end{proof}

The following corollary will play an important role in the next section and has interest in its own right.
\begin{cor} \label{cor-red}
    Let $[a,b]$ ($a \ge b$) be an imaginary positive reduced root of $\mathcal H(m)$ with level $L \ge 2$, and $\bc_k=(a_{k,1}, a_{k,2}, \dots , a_{k, d_k})$ its canonical sequences for $2\leq k\leq L$.

(1) If $k=2l+2$, then
    \begin{equation} \label{sabeven} s^{a \times b} =(13 2)^{l}(21)^{a_{k,1}}(31) (21)^{a_{k,2}}(31) \cdots (21)^{a_{k,d_k}}(31)(231)^{l} .\end{equation}

(2) If $k=2l+1$, then
    \begin{equation} \label{sabodd} s^{a \times b} =(1 32)^{l}(23)^{a_{k,1}+1}(21) (23)^{a_{k,2}+1}(21) \cdots (23)^{a_{k,d_k}+1}(21)(231)^{l} .\end{equation}
Moreover, we have
\[ (231)^l s^{a \times b} (132)^l =  s^{\tilde a \times \tilde b} \] for $[\tilde a, \tilde b]$ whose first canonical sequence is $(a_{k,1}+1, \dots, a_{k,d_k}+1)$ with level $L-k+1$. 
        \end{cor}

\begin{remark}
The part (2) is related to Lemma 3.3 (2) of \cite{LL1}, which connects the Weyl group action on the set of roots of $\mathcal H(m)$ with the set of rigid reflections. Though the above corollary is not enough to prove injectivity, one may find that the expressions look more natural than those obtained in the next section after the reduction is completed.
\end{remark}

        \begin{proof}
            Assume that $k=2l+2$. If $k=2$ then \eqref{sabeven} is nothing but \eqref{sabH}. If $k \ge 4$, we obtain from \eqref{sabH}
                \begin{align*} s^{a \times b}&=13(213)^{l-1}      (12)^{a_{k,1}} (2312) (12)^{a_{k,2}} (2312) \cdots (12)^{a_{k,d_k} }(2312)(312)^{l-1}31\\& = 13(213)^{l-1}   (21) \  (12)(12)^{a_{k,1}} (23) \ (12)(12)^{a_{k,2}}(23)\ (12) (12)^{a_{k,3}}\\ & \phantom{LLLLLLLLLLLLLLLLLLLLLLLLLLLLL} \cdots (23)\ (12)(12)^{a_{k,d_k}} (2312)(312)^{l-1}31 \\&=  (132)^{l}    1  \ (12)^{a_{k,1}}(13) \ (12)^{a_{k,2}}(13) \cdots (12)^{a_{k,d_k}}(13) 1(231)^{l} \\&=  
        (13 2)^{l}(21)^{a_{k,1}}(31) (21)^{a_{k,2}}(31) \cdots (21)^{a_{k,d_k}}(31)(231)^{l}.
        \end{align*}
        The case $k=2l+1$ is similar and we omit the details. The last assertion is clear from the definitions.
        \end{proof}

        Though Proposition \ref{lem:notreduced} and Corollary \ref{cor-red} provide reductions of the initial expression of $s^{a \times b}$, it is not sufficient to prove injectivity of the map $[a,b] \mapsto s^{a \times b}$. In Section \ref{reductions}, we will further reduce $s^{a \times b}$ to its standard word to show injectivity.
        
\subsection{Dichotomy between the cases $a>b$ and $a<b$} In this subsection, we show that $s([a,b])$ with $a<b$ cannot be equal to any of $s([a_1, b_1])$ with $a_1 > b_1$ in $W(m)$. Consequently, since the roles of $1$ and $3$ can be interchanged, it will be enough to establish Theorem \ref{thm-main} for a proof of Theorem \ref{thm-in}.

First we recall a result from \cite{LL1}. Let $\sigma_1$ and $\sigma_2$ be the simple reflections of $\mathcal H(m)$ associated with the simple roots $[1,0]$ and $[0,1]$, respectively. Then they act on $[a,b] \in \mathbb Z^2$ in the usual way by
\[ \sigma_1[a,b]=[-a+mb, b] \qquad \text{and} \qquad \sigma_2[a,b]=[a,-b+ma]. \]
\begin{lem} \label{lem-ll1} \cite[Lemma 3.3 (3)]{LL1}
Assume that $[a,b] \in \mathcal P^+$ with $a \ge b$, and write $[c,d]=\sigma_1 \sigma_2 [a,b]$. Then we have \[ s_3s_2s_1 s([a,b]) s_1s_2s_3 = s([c,d]). \]
\end{lem}

Next we need to compute explicitly $(s_2s_3)^n s_2s_1$ for $n \ge 0$. Let \[ x=2 \cos(\pi/m).\] Then the matrices of $s_i$ with respect to $\{\alpha_1, \alpha_2, \alpha_3\}$ are given by
\begin{equation} \label{eqn-mms} s_1= \begin{bmatrix} -1 & x & 2 \\ 0&1&0 \\ 0&0&1 \end{bmatrix}, \qquad s_2= \begin{bmatrix} 1&0&0 \\ x&-1&x \\ 0&0&1\end{bmatrix}, \qquad s_3= \begin{bmatrix} 1&0&0 \\ 0&1&0 \\ 2&x&-1 \end{bmatrix}. \end{equation}
Write \begin{equation} \label{eqn-s2s} (s_2s_3)^n s_2s_1= [\tau^{(n)}_{i,j}]_{1 \le i,j \le 3} \qquad \text{ for }n\ge 0.\end{equation}
Then we have
\[  \tau^{(n)}_{1,1}=-1, \qquad \tau^{(n)}_{1,2}=x, \qquad  \tau^{(n)}_{1,1}=2 
\qquad \text{ for } n\ge 0, \] and obtain the recursive relations
\[ \begin{bmatrix} \tau^{(n+1)}_{2,j} \\ \tau^{(n+1)}_{3,j} \end{bmatrix} =   A \begin{bmatrix} \tau^{(n)}_{2,j} \\ \tau^{(n)}_{3,j} \end{bmatrix} + B_j \quad (j=1,2,3),  \quad [\tau^{(0)}_{i,j}]_{\substack{i=2,3\\ j=1,2,3}} = \begin{bmatrix}-x&x^2-1&3x\\0&0&1 \end{bmatrix},  \] 
where we set 
\[ A=\begin{bmatrix} x^2-1 & -x \\ x &-1 \end{bmatrix}, \quad B_1=\begin{bmatrix} -3x\\-2 \end{bmatrix}, \quad  B_2= \begin{bmatrix} 3x^2\\2x\end{bmatrix}, \quad B_3 = \begin{bmatrix} 6x\\4 \end{bmatrix}. \]

Consider $x$ as a variable for the time being, and write $A^n = \begin{bmatrix} f_n(x) & -g_n(x) \\ g_n(x) & -f_{n-1}(x) \end{bmatrix}$ for $n \ge 1$. Then it follows from the definition that  
\begin{align} \label{eqn-gn1}  g_{n+1} &=xf_n -g_n = (x^2-1)g_n -x f_{n-1},  &  f_n&=xg_n-f_{n-1}, \\ \label{eqn-gn2} xg_{n+1}&=(x^2-1)f_{n}-f_{n-1}, & f_0&=1, \quad g_0=0, \end{align} and we obtain
\[  f_n(x)=\sum_{k=0}^n (-1)^{n-k} \binom{n+k}{n-k} x^{2k} \quad \text{ and } \quad g_n(x)=xU_{n-1}(\tfrac{x^2} 2 -1), \] 
where $U_n(x)$ are the Chebyshev polynomials of the second kind. It is well known that the roots of $U_n(x)$ are $x=\cos \left ( \frac k {n+1} \pi \right )$, $k=1,2, \dots , n$. Thus $g_n(2) >0$  and the largest roots of $g_n$ in the interval $[0,2]$ are 
\begin{equation} \label{eqn-x22}  x=2 \cos \left (\frac {\pi} {2n} \right ) \quad \text{ for } n \ge 1. \end{equation}
Similarly, $f_n(2) >0$  and the largest roots of $f_n$ in the interval $[0,2]$ are 
\begin{equation} \label{eqn-x23}  x=2 \cos \left (\frac {\pi} {2n+1} \right ) \quad \text{ for } n \ge 1. \end{equation}
Thus when $x= 2 \cos(\pi/m)$ and $n=\lceil \frac m2 \rceil$, we have
\begin{equation} \label{eqn-cos}
f_{k}(x) \ge 0 \qquad \text{ for } k=1,2, \dots , n-1.  
\end{equation}

\begin{lem} \label{lem-cheby} For $n \ge 2$, we have
\begin{align*}
\tau^{(n)}_{3,1}& = -(x^2+2)f_{n-1}-2-\sum_{k=1}^{n-2}h_k= 
-xg_n-2 - \sum_{k=1}^{n-1}  h_k,\\
\tau^{(n)}_{3,2} & = \frac {x^4+x^2+1} x f_{n-1} + \frac 1 x f_{n-2} + 2x + x \sum_{k=1}^{n-2} h_k= (x^2-1)g_n+2x+x \sum_{k=1}^{n-1}h_k,\\
\tau^{(n)}_{3,3}& = (3x^2+2)f_{n-1}+5f_{n-2}+2f_{n-3}+4+ 2 \sum_{k=1}^{n-3}h_k\\&=
3x g_n+ \frac 1 x (5g_n+7g_{n-1}+2g_{n-2}) +4 + 2   \sum_{k=1}^{n-2} h_k ,
\end{align*}
\[ \tau_{3,1}^{(1)}=-x^2-2, \qquad \tau_{3,2}^{(1)}=x^3+x, \qquad \tau_{3,3}^{(1)}= 3x^2+3, \]
where we set $h_k=3f_k+f_{k-1}= \frac 1 x (3g_{k+1}+4g_k+g_{k-1})$ and $f_{-1}=-3$.
\end{lem}

\begin{proof}
The assertions follow from  \eqref{eqn-gn1}, \eqref{eqn-gn2} and
\[ \begin{bmatrix} \tau^{(n)}_{2,j} \\ \tau^{(n)}_{3,j} \end{bmatrix} =   A^n \begin{bmatrix} \tau^{(0)}_{2,j} \\ \tau^{(0)}_{3,j} \end{bmatrix} + (A^{n-1}+ A^{n-2}+ \cdots +I) B_j \qquad (j=1,2,3). \]
\end{proof}

Now we will prove the following proposition.
\begin{prop} \label{prop-im}
Assume that $[a,b]$ is a positive reduced root of $\mathcal H(m)$ with $a> b$. Let $p \alpha_1 + q \alpha_2 + r \alpha_3$ be the positive root of $W(m)$ associated with $s([a,b])$. Then we have \[  0  \le p < r. \] 
\end{prop}

\begin{proof}
Let $x=2 \cos(\pi/m)$. Then we obtain from \eqref{eqn-mms} \[s_3s_2s_1= \begin{bmatrix}-1&x&2\\-x&x^2-1&3x\\-x^2-2&x^3+x &3x^2+3 \end{bmatrix}.\] 
Suppose that $[c,d]=\sigma_1 \sigma_2[a,b]$. Then $[c,d]$ is an imaginary positive reduced root of $\mathcal H(m)$, since the set of positive imaginary roots are invariant under the Weyl group action. Let $p \alpha_1+q\alpha_2+r\alpha_3$ be the root associated with $s([a,b])$, and assume $0 \le p < r$. (Note that $q \ge 0$ from the assumption.) Then, by Lemma \ref{lem-ll1} and the computation of $s_3s_2s_1$, the root associated with $s([c,d])$ is equal to $p'\alpha_1+ q'\alpha_2+ r'\alpha_3$, where  \[p'= -p+qx+2r, \quad q'=-px+q(x^2-1) +3rx, \quad r'=-p(x^2+2)+q(x^3+x)+3r(x^2+1).\] We see that $0 \le p'< r'$, and $[c,d]$ also satisfies the assertion of the proposition. Thus it is enough to consider a set of representatives $[a,b]$ with $a > b$ from the orbits of the Weyl group action on the set of positive roots. 

For the real roots $[a,b]$ ($a>b$), we can take $[F_2,F_1]$ as a representative. Then it follows from Lemma \ref{sFnFn1} that $s([F_2,F_1])=s_3$. Since the associated root is simply $\alpha_3$, the assertion of the proposition holds in this case.

For the imaginary roots $[a,b]$ ($a>b$), such a set of representatives is given by the condition  
\begin{equation} \label{eqn-km}  0 \le \frac {2a} m  \le b < a . \end{equation} (See \cite{KaMe} for more details.)

For the rest of the proof, we assume that $[a,b]$ satisfies  \eqref{eqn-km}. Then the first canonical sequence of $[a,b]$ occurs with $N_1 \le \lceil \frac m 2  \rceil -1$, and the root associated with $s([a,b])$ is equal to  
\[ S_lS_{l-1} \cdots S_1 \alpha , \] where $S_i =(s_2s_3)^{n_i} (s_2s_1)$ with $1 \le n_i \le \lceil \frac m 2  \rceil$, $i=1,2, \dots , l$ and $\alpha= (s_2s_3)^{n_0}\alpha_2$  for $1 \le n_0 \le \lceil \frac m 2  \rceil-1$, or $(s_2s_3)^{n_0}s_2\alpha_3$  for $0 \le n_0 \le \lceil \frac m 2  \rceil-2$, or $(s_2s_3)^{n_0}s_2\alpha_1$ for $ n_0= \lceil \frac m 2  \rceil-1$. 

Write $\alpha=p_0 \alpha_1 + q_0 \alpha_2 + r_0 \alpha_3$. We claim that $0 \le p_0 <r_0$. 

i) If $\alpha = (s_2s_3)^{n_0} \alpha_2$ for $1 \le n_0 \le \lceil \frac m 2  \rceil-1$, then $p_0=0$ and $r_0=g_{n_0}(x)$. Recall from \eqref{eqn-x22} that the largest root of $g_{n_0}$ is $2\cos(\frac \pi {2n_0})$. Since $x= 2\cos( \frac \pi m) > 2\cos(\frac \pi {2n_0})$, we have $r_0>p_0=0$ as claimed.

ii) If $\alpha = (s_2s_3)^{n_0} s_2\alpha_3$ for $0 \le n_0 \le \lceil \frac m 2  \rceil-2$, then $p_0=0$ and $r_0=f_{n_0}(x)$. Recall from \eqref{eqn-x23} that the largest root of $f_{n_0}$ is $2\cos(\frac \pi {2n_0+1})$. Since $x= 2\cos( \frac \pi m) > 2\cos(\frac \pi {2n_0+1})$, we have $r_0>p_0=0$ as claimed.

iii) If $\alpha = (s_2s_3)^{n_0} s_2\alpha_1$ for $n_0=\lceil \frac m 2  \rceil-1$, then $p_0=1$ and $r_0=-\tau^{n_0}_{3,1}$. It follows from Lemma \ref{lem-cheby} and \eqref{eqn-cos} that $r_0>p_0 \ge 0$ as claimed.

For induction, write $S_{l-1} \cdots S_1\alpha=p \alpha_1 + q \alpha_2 + r \alpha_3$ and assume $0 \le p < r$. Then, using \eqref{eqn-s2s},  we compute $S_l S_{l-1} \cdots S_1\alpha=p' \alpha_1 + q' \alpha_2 + r' \alpha_3$ to obtain 
\[ p'= -p+xq+2r \qquad \text{ and } \qquad r'=\tau_{3,1}^{(n_l)} p +\tau_{3,2}^{(n_l)} q + \tau_{3,3}^{(n_l)} r. \]
Clearly, $p' \ge 0$. 
It follows from Lemma \ref{lem-cheby} and \eqref{eqn-cos} that  $\tau_{3,2}^{(n_l)}-x >0$,  $\tau_{3,3}^{(n_l)}-2 >0$ and 
\[ (\tau_{3,1}^{(n_l)}+1)  + (\tau_{3,3}^{(n_l)}-2) >0 .\] 
Then we have 
\begin{align*} r'-p' &= (\tau_{3,1}^{(n_l)}+1) p + (\tau_{3,2}^{(n_l)}-x) q + (\tau_{3,3}^{(n_l)}-2) r \\ & \ge (\tau_{3,1}^{(n_l)}+1) p + (\tau_{3,3}^{(n_l)}-2) r > (\tau_{3,1}^{(n_l)}+1) p + (\tau_{3,3}^{(n_l)}-2) p \ge 0 .\end{align*} 
Therefore $r' >p'$, and this completes the proof.
\end{proof}

The following corollary will be used in the proof of Theorem \ref{thm-in}. 
\begin{cor} \label{cor-agb}
Assume that $[a,b]$, $[a', b']$ are positive reduced roots of $\mathcal H(m)$ such that $a \ge b$ and $a'< b'$, respectively. Then $s([a,b])$  cannot be equal to $s([a', b'])$ in $W(m)$. 
\end{cor}

\begin{proof}
If the positive root associated to $s([a,b])$ is $p\alpha_1+ q \alpha_2 +r \alpha_3$, then we have $p \le r$ by Proposition \ref{prop-im} and from the fact that $[1,1]$ is the only positive reduced root with $a=b$ and $s([1, 1])=s_2$. If the positive root associated to $s([a',b'])$ is $p'\alpha_1+ q' \alpha_2 +r' \alpha_3$, then we have $p' > r'$ by interchanging the roles of $s_1$ and $s_3$. Thus $s([a,b])$ cannot be equal to $s([a',b'])$. 
\end{proof}

\medskip

        \section{Reduction to standard words} \label{reductions}

        In this section, we reduce each $s^{a \times b}$ to its standard word in $W(m)$, starting with an expression in Proposition \ref{lem:notreduced}, and prove Theorem \ref{thm-main} by showing all the standard words are distinct.
        \medskip

        For $k \in \mathbb Z_{\ge 2}$, 
        define 
        \begin{equation*}
        \begin{split}
        \mathcal{S}(2k-1)=&\{s_1^2-e,s_2^2-e,s_3^2-e, (s_1 s_2)^{k-1}s_1 - (s_2s_1)^{k-1}s_2 , 
         (s_2 s_3)^{k-1}s_2 - (s_3s_2)^{k-1}s_3\},\\
        \mathcal{S}(2k)=&\{s_1^2-e,s_2^2-e,s_3^2-e, (s_1s_2)^k - (s_2s_1)^k,\\
        &\quad\quad\quad (s_2s_3)^k -(s_3s_2)^k,  (s_1s_2)^{k-1} s_1 (s_3s_2)^{k}- (s_2 s_1)^{k} s_3 (s_2s_3)^{k-1} \}.
        \end{split}
        \end{equation*}

            It will be shown in Propositions \ref{prop:s(m)odd} and \ref{prop:s(m)even} that  $\mathcal{S}(m)$ is a \GS basis of $W(m)$ for $m \ge 3$. Thus we take $\mathcal S(m)$-standard words or monomials (see Definition \ref{def-standard}) as standard expressions of the elements of $W(m)$.

            \medskip

            In this section, as in Proposition \ref{lem:notreduced}, we write only subscripts of simple reflections when we express elements in $W(m)$, and the identity element of $W(m)$ will be denoted by $e$.
            Before delving into general cases, let us look at a simple example. 
            \begin{example} \label{exst}
            Suppose $[a,b]=[5,3]$ and $m=3$. By definition an $\mathcal S(3)$-standard word cannot have any of $11,22, 33, 121, 232$ as a subword. Clearly, the level of $[5,3]$ is 1.  By Example \ref{ex:[5,3]} (1) and Proposition \ref{lem:notreduced}, we have
            \[s^{5\times 3}=(23)^2(21)(23)^2(21)(23)^1(21),\] which is not $\mathcal S(3)$-standard because $232$ is a subword. Using the relations $232=323, 22=e, 33=e$, we obtain 
            \[ s^{a \times b} =(31)(31)(3231) ,\]
            which is $\mathcal S(3)$-standard.
            \end{example}

Recall that the positive reduced real roots $[a,b]$ of $\mathcal H(m)$ with $a>b$ are precisely $[F_n, F_{n-1}]$, $n \ge 2$. By simple investigations, one can see that the following lemma is true.

\begin{lem} \label{lem-real}
The expressions of $s^{F_n \times F_{n-1}}$ in Lemma \ref{sFnFn1} (2) are $\mathcal S(m)$-standard and all distinct for $n\ge 2$.
\end{lem}

            In what follows we obtain $\mathcal S(m)$-standard words for $s^{a \times b}$, where $[a,b]$ is a positive imaginary reduced root of $\mathcal H(m)$. Let $L$ be the level of $[a,b]$. The canonical sequences $\bc_k=(a_{k,1}, a_{k,2}, \dots , a_{k, d_k})$ and the numbers $N_k$ are defined in Definition \ref{def:a_n}.

            \medskip

            Define $N:=N_L-\delta_{L,1}$ and 
            $\ell := \lfloor m/2 \rfloor$, where $\delta_{i,j}$ is Kronecker's delta. According to the values of $N$, let $w_1$ and $w_2$ be the elements of $W(m)$ defined by
            the following table.
            \begin{equation}\label{tab-w} \text{
            \begin{tabular}{|c||c|c|}
            \hline
            $N$          & $w_1$                       & $w_2$  \\ \hline \hline
            $N \le \ell-3$  & $(23)^{N+2}(21)$    &   $(23)^{N+1}(21)$            \\ \hline
            $\ell-2$ & $(32)^{m-\ell-1}(31)$     & $(23)^{\ell-1}(21)$               \\ \hline
            $ \ell-1 \le N \le m-3$   & $(32)^{m-N-3}(31)$         & $(32)^{m-N-2}(31)$            \\ \hline
            $m-2$    & $21$        & $31$                 \\ \hline
            \end{tabular}} 
            \end{equation}

            \begin{lem}[level $L=1$] \label{lem:1standodd}
            Assume that the level $L$ of $[a,b]$ is $1$. 
            Unless $m=3$, $N_1=2$ and $\bc_1$ is of type $-$,
            the following expression of $s^{a \times b}$ is $\mathcal{S}(m)$-standard:
            \begin{equation*}
            s^{a \times b}=
            \begin{cases}
            w_1^{a_{2,1}}w_2\cdots w_1^{a_{2,d_2}}w_2 &\text{if $\bc_1$ is of type $+$,}\\
            w_1w_2^{a_{2,1}}\cdots w_1w_2^{a_{2,d_2}} &\text{if $\bc_1$ is of type $-$,}\\
            w_1w_2  &\text{if $\bc_1$ is of type $=$,}\\
            w_2, &\text{if $\bc_1$ is of type $0$.}
            \end{cases}
            \end{equation*}
            The case when $m=3$, $N_1=2$ and $\bc_1$ is of type $-$ is covered in Lemma \ref{lem:m=3 level1}.
            \end{lem}

        \begin{example}
        Continuing with Example \ref{exst}, assume $[a,b]=[5,3]$ and $m=3$. Then $N=N_1-1=0$, $\ell=1$, $\rho_1=\frac 2 3$ and $\bc_2 =(2)$ from Example \ref{ex:[5,3]} (1). Thus we have $w_1=31$ and $w_2=3231$, and \[ s^{a\times b} = w_1^2 w_2 =(31)(31)(3231).\] This coincides with the standard word in Example \ref{exst}.  
        \end{example}

            \begin{proof}[Proof of Lemma \ref{lem:1standodd}]
                By Proposition \ref{lem:notreduced}, we obtain
                \begin{equation} \label{s23a11} s^{a \times b} = (23)^{a_{1,1}}(21) (23)^{a_{1,2}} (21) \cdots (23)^{a_{1,d_1}} (21). \end{equation}
            Note that the $\mathcal{S}(m)$-standard word of $(23)^s(21)$ is  equal to 
                \begin{equation}\label{2321} \begin{aligned}
                (23)^s(21),\quad&\text{if $s\le \ell-1$}, \text{ or} \\
                (32)^{m-s-1}(31),\quad&\text{if $\ell \leq s <m$}.
                \end{aligned}   \end{equation} Since $N=N_1-1$, we have four different cases according to \eqref{tab-w}.

        {\it Case 1}: $1\leq N_1 \le \ell-2$.
                Since $a_{1,i} = N_1$ or $N_1+1$, the $\mathcal{S}(m)$-standard word of $(23)^{a_{1,i}}(21)$ is equal to itself for all $1\leq i \leq d_1$. Since there are no relations involving $w_1 = (23)^{N_1+1}(21)$ and $w_2=(23)^{N_1}(21)$ in the expression \eqref{s23a11}, it is already standard. Thus, if $\bc_1$ is of type $+$ (resp. $-$), we obtain \[s^{a \times b} = w_1^{a_{2,1}}w_2\cdots w_1^{a_{2,d_2}}w_2 \quad \text{(resp. $w_1w_2^{a_{2,1}}\cdots w_1w_2^{a_{2,d_2}}$)}  \] from the definition of $\bc_2$.

             Suppose $\bc_1$ is of type $=$. Since $\frac{a}{b}=N_1+\frac{1}{2}$, we obtain $b=2$ and
             \begin{equation*}
             s^{a\times b}=(23)^{N_1+1}(21)(23)^{N_1}(21)=w_1w_2.
             \end{equation*}
             Assume $\bc_1$ is of type $0$. Then $\frac{a}{b}=N_1$. Since $a$ and $b$ are coprime, $a=b=1$ and $N_1=1$.
            By the definition of $s^{a\times b}$, we obtain
            \begin{equation*}
            s^{a\times b}=(23)(21)=w_2.
            \end{equation*}

            {\it Case 2}: $N_1=\ell-1$.
                Recall $a_{1,i} = N_1$ or $N_1+1$. The word $(23)^{N_1}(21)$ is $\mathcal S(m)$-standard, while $(23)^{N_1+1}(21)$ is reduced to $(32)^{m-\ell-1}(31)$. These words are $w_2$ and $w_1$ respectively.
            Moreover, there are no additional relations between $(23)^{N_1}(21)$ and  $(32)^{m-\ell-1}(31)$.
                Hence, we obtain the desired expressions of $s^{a \times b}$ similarly to {\it Case 1.}

            {\it Case 3}: $\ell \le N_1 \le m-2$.
                The $\mathcal{S}(m)$-standard words of $(23)^{N_1}(21)$ and $(23)^{N_1+1}(21)$ are $w_2=(32)^{m-N_1-1}(31)$ and $w_1=(32)^{m-N_1-2}(31)$ respectively. Moreover, there are no additional relations between $w_1$ and $w_2$. Hence, we obtain the desired expressions of $s^{a \times b}$ similarly to {\it Case 1}.

            {\it Case 4}: $N_1=m-1$.
            Note that
            \begin{equation*}
            \begin{split}
            &(23)^{m-1}(21)=(32)(21)=31,\\
            &(23)^{m}(21)=21.
            \end{split}
            \end{equation*}
                As in {\it Case 1}, if $\bc_1$ is of type $0$, then $s^{a\times b}$ is equal to $w_2=31$, and if $\bc_1$ is of type $=$, then 
             $s^{a\times b}$ is equal to $w_1w_2=(21)(31)$.
                If $\bc_1$ is of type $-$, then \[ \bc_1=(m,(m-1)^{a_{2,1}}, m, (m-1)^{a_{2,2}}, \dots , m, (m-1)^{a_{2,d_2}}),\] where we write
            $(m-1)^s= \underbrace{m-1,m-1,\ldots,m-1}_{s\text{ times}}$, and  $s^{a\times b}$ in \eqref{s23a11} becomes equal to
            \begin{equation} \label{2131a21}
                (21)(31)^{a_{2,1}}(21)(31)^{a_{2,2}}\cdots(21)(31)^{a_{2,d_2}}=w_1w_2^{a_{2,1}}\cdots w_1w_2^{a_{2,d_2}} .
            \end{equation} If $m > 3$, then this expression is standard and we obtain the desired form. If $m=3$, then it is not standard because of the subword $121$ and this case is covered in Lemma \ref{lem:m=3 level1}.
        Finally, $\bc_1$ cannot be of type $+$ by Corollary \ref{cor-cannot} (1). It completes the proof.
            \end{proof}

        Now we move on to higher levels. As before, define $\ell := \lfloor m/2 \rfloor$. 
        According to the values of $N_L$, let $v_1$ and $v_2$ be defined by the following table.
        \begin{equation}\label{tab-v} \text{
        \begin{tabular}{|c||c|c|}
        \hline
        $N_L$               & $v_1$                       & $v_2$  \\ \hline \hline
        $N_L \le \ell-2$  & $(12)^{N_L+1}(13)$    &   $(12)^{N_L}(13)$            \\ \hline
        $\ell-1$ & $(21)^{m-\ell-1}(23)$     & $(12)^{\ell-1}(13)$               \\ \hline
        $ \ell \le  N_L \le m-2$   & $(21)^{m-N_L-2}(23)$         & $(21)^{m-N_L-1}(23)$            \\ \hline
        \end{tabular}
            } \end{equation}
            Define \[ x= (132)^{\lfloor \frac{L-2}2 \rfloor} 1,\quad x^{-1}= 1(231)^{\lfloor \frac{L-2}2 \rfloor} \quad \text{ and } \quad y=\begin{cases} (132)^{\frac{L-4}2}13 & \text{ if $L\ge 4$}, \\  \hat 2 & \text{ if $L=2$,}\end{cases} \]
        where $\hat 2$ means $2$ if the following letter is different from $2$, or removing the following letter  $2$ otherwise. For example,   $\hat 2 13=213$ and $\hat 2 23=3$. If $L=3$, we do not need to define $y$. 

            \begin{lem} [level $L\ge 2$] \label{lem:generalandodd}
            Assume that the level of $[a,b]$ is $L \ge 2$. 
             
            \begin{itemize}
            \item[(1)] Suppose that $L$ is even. Unless $m=3,4,5$ and $N_L=m-2$, the $\mathcal S(m)$-standard word of $s^{a\times b}$ is equal to 
            \begin{equation*}
            \begin{cases}
            yv_2v_1^{a_{L+1,1}-1}v_2[v_1^{a_{L+1,2}}v_2\cdots v_1^{a_{L+1,d_{L+1}}}v_2]x^{-1} &\text{if $\bc_L$ is of type $+$,}\\
            yv_2^{a_{L+1,1}+1}[v_1v_2^{a_{L+1,2}}\cdots v_1v_2^{a_{L+1,d_{L+1}}}]x^{-1} &\text{if $\bc_L$ is of type $-$,}\\
            yv^2_2x^{-1}  &\text{if $\bc_L$ is of type $=$,}\\
            y(12)^{N_L-1}13x^{-1} &\text{if $\bc_L$ is of type $0$ and $N_L \leq \ell$,}\\
            xv_2x^{-1} &\text{if $\bc_L$ is of type $0$ and $N_L \ge \ell+1$.}\\
            \end{cases}
            \end{equation*} Here the expression inside $[\ ]$ is void if $d_{L+1}= 1$.
            The case when $m=3,4,5$ and $N_L=m-2$ is considered in Lemma \ref{lem:m=3,4,5 even L>1}. 

            \item[(2)] Suppose that $L$ is odd. Unless $m=3$ and $\bc_L$ is of type $-$,  the $\mathcal S(m)$-standard word of $s^{a\times b}$ is equal to
            \begin{equation*}
            \begin{cases}
            xw_2w_1^{a_{L+1,1}-1}w_2[w_1^{a_{L+1,2}}w_2\cdots w_1^{a_{L+1,d_{L+1}}}w_2] 23x^{-1} &\text{if $\bc_L$ is of type $+$,}\\
            xw_2^{a_{L+1,1}+1}[w_1w_2^{a_{L+1,2}}\cdots w_1w_2^{a_{L+1,d_{L+1}}}]23x^{-1} &\text{if $\bc_L$ is of type $-$,}\\
            xw_2^223x^{-1} &\text{if $\bc_L$ is of type $=$,}\\
            x(23)^{N_L}2123x^{-1}  &\text{if $\bc_L$ is of type $0$ and $N_L \le \ell -1$,}\\
            x32w_223x^{-1} &\text{if $\bc_L$ is of type $0$ and $N_L \ge \ell$.}
            \end{cases}
            \end{equation*}
        Here $w_1$ and $w_2$ are given by \eqref{tab-w} as before and the expression inside $[\ ]$ is void if $d_{L+1}= 1$.	The case when $m=3$ and $\bc_L$ is of type $-$ is dealt with in Lemma \ref{lem: m=3 odd L>1}.
            \end{itemize}
             
            \end{lem}

        \begin{example}
            Let $m=6$ and $[a,b]=[73,13]$. Then \[ \bc_1=(6, 6, 5, 6, 6, 5, 6, 5, 6, 6, 5, 6, 5),\quad \bc_2= (2, 2, 1, 2, 1),\quad \bc_3=( 2, 1),\] and $N_1=5, N_2=1, \rho_1=\frac 8{13}, \rho_2=\frac 3 5$. The level $L$ is equal to $2$ and $\bc_2$ is of type $+$. Note that $v_1=121213$ and $v_2=12134$ from \eqref{tab-v}. By Lemma \ref{lem:generalandodd} (1), we obtain \[ s^{73 \times 13} =y v_2 v_1 v_2 v_1 v_2 x^{-1}=2(1 2 1 3) (1 2 1 2 1 3) (1 2 1 3) (1 2 1 2 1 3) (1 2 1 3) 1. \]

        \end{example}

        \begin{proof}[Proof of Lemma \ref{lem:generalandodd}]
        Suppose $L=2g+2 \ge 2$. By Corollary \ref{cor-red} (1), we have
        \begin{equation*}
        \begin{split}
        s^{a\times b}&=(132)^g(21)^{a_{ L,1}}(31) (21)^{a_{ L,2}}(31) \cdots (21)^{a_{L, d_L}} (31)(231)^g\\
        &=(132)^g1(12)^{a_{L,1}}(13)\cdots(12)^{a_{L,d_L}}(13)1(231)^g\\&=x(12)^{a_{L,1}}(13)\cdots(12)^{a_{L,d_L}}(13) x^{-1}.
        \end{split}
        \end{equation*}

        We have $1\leq a_{L,i} \leq m-1$ for $1\leq i \leq d_L$ by Lemma \ref{lem:a_{n,i}} (2) and Lemma \ref{lem:range}.
        Note that the $\mathcal{S}(m)$-standard word of $(12)^s(13)$ for $1\leq s \le  m-1$ is 
            \begin{equation} \label{1213}
        \begin{cases}
        (12)^s(13)&\text{ if $s\le \ell-1$},\\
         (21)^{m-s-1}(23)&\text{ if $\ell \leq s \le m-1$}.
         \end{cases}
         \end{equation}

        We apply the same argument as in the proof of Lemma \ref{lem:1standodd} to  \begin{equation} \label{x-1sab} x^{-1} s^{a \times b} x = (12)^{a_{L,1}}(13)\cdots(12)^{a_{L,d_L}}(13) \end{equation}
        with \eqref{2321} replaced by \eqref{1213} and obtain
        \begin{equation*}
            x^{-1}s^{a \times b}x=
            \begin{cases}
            v_1^{a_{L+1,1}}v_2\cdots v_1^{a_{L+1,d_{L+1}}}v_2 &\text{if $\bc_L$ is of type $+$,}\\
            v_1v_2^{a_{L+1,1}}\cdots v_1v_2^{a_{L+1,d_{L+1}}} &\text{if $\bc_L$ is of type $-$,}\\
            v_1v_2  &\text{if $\bc_L$ is of type $=$,}\\
            v_2, &\text{if $\bc_L$ is of type $0$.}
            \end{cases}
            \end{equation*}

        After conjugating both sides by $x$, we obtain an expression of $s^{a \times b}$. Assume it is not the case that $m=3,4,5$ and $N_L=m-2$. If $\bc_L$ is of type $+$, $-$ or $=$, we apply $xv_1=yv_2$ to the leftmost part of the expression and obtain the desired standard word. For $\bc_L$ of type $0$, the word $xv_2x^{-1}$ is standard if $N_L \ge \ell+1$; otherwise $xv_2 x^{-1}$ reduces to $y (12)^{N_L-1}13 x^{-1}$, which is standard. 

        In the case when $m=3,4,5$ and $N_L=m-2$, the standard word of $(12)^{a_{L,i}}(13)$ in \eqref{x-1sab} is $2123$ or $23$ since $a_{L,i}=m-2$ or $m-1$ for $i=1, 2, \dots , d_L$. For example, if $\bc_L$ is of type $-$, an expression of $s^{a \times b}$ is equal to    \[ x(23)(2123)^{a_{L+1,1}}(23)(2123)^{a_{L+1,2}} \cdots (23)(2123)^{a_{L,d_{L+1}}} x^{-1}. \] Since this expression has a subword $23232$, it is not standard exactly when $m=3,4,5$. This case will be handled in Lemma \ref{lem:m=3,4,5 even L>1}.  

        Now suppose $L=2g+1 \ge 3$. By Corollary \ref{cor-red} (2), we have
            \begin{align*} s^{a \times b} &=(1 32)^{g}(23)^{a_{L,1}+1}(21) (23)^{a_{L,2}+1}(21) \cdots (23)^{a_{L,d_L}+1}(21)(231)^{g} \\  &=x(32) (23)^{a_{L,1}+1}(21) (23)^{a_{L,2}+1}(21) \cdots (23)^{a_{L,d_L}+1}(21)(23) x^{-1} .\end{align*}

        We apply the same argument as in the proof of Lemma \ref{lem:1standodd} to  \begin{equation} \label{x-1sabx} (23)x^{-1} s^{a \times b} x(32) =  (23)^{a_{L,1}+1}(21) (23)^{a_{L,2}+1}(21) \cdots (23)^{a_{L,d_L}+1}(21) \end{equation}
        and obtain
        \begin{equation}\label{23x-1}
        (23)x^{-1}s^{a \times b}x(32)=	\begin{cases}
            w_1^{a_{L+1,1}}w_2\cdots w_1^{a_{L+1,d_{L+1}}}w_2 &\text{if $\bc_L$ is of type $+$,}\\
            w_1w_2^{a_{L+1,1}}\cdots w_1w_2^{a_{L+1,d_{L+1}}} &\text{if $\bc_L$ is of type $-$,}\\
            w_1w_2  &\text{if $\bc_L$ is of type $=$,}\\
            w_2 &\text{if $\bc_L$ is of type $0$.}
            \end{cases}
        \end{equation}
        Here the shift of $\bc_L=(a_{L,1}, \dots, a_{L,d_L})$ by $1$ in the exponents of \eqref{x-1sabx} is reflected in the definition of $N=N_L -\delta_{L,1}$, and we still get $\bc_{L+1} =(a_{L+1,1}, \dots , a_{L+1, d_{L+1}})$ in the exponents of \eqref{23x-1} since the shift of $\bc_L$ by $1$ does not change how many times a number repeats in the sequence.
         
        We see from \eqref{tab-w} that \[  (32) w_1 =w_2 \qquad \text{ for }  1 \le N_L \le m-2 .\] For example, when $N_L=\ell -2$, we have
        \[ (32)w_1 = (32)^{m-\ell}(31) = (23)^\ell (31) = (23)^{\ell -1} (23)(31) = (23)^{\ell -1} (21) =w_2.\]

        Assume it is not the case that $m=3$ and $\bc_L$ is of type $-$. After conjugating both sides of \eqref{23x-1} by $x(32)$, we obtain an expression of $s^{a \times b}$. If $\bc_L$ is of type $+$, $-$ or $=$, we apply $x(32)w_1=xw_2$ to the leftmost part of the expression and obtain the desired standard word. For $\bc_L$ of type $0$, the word $x(32)w_2(23)x^{-1}$ is standard if $N_L \ge \ell$; otherwise $x(32)w_2 (23)x^{-1}$ reduces to the standard word $x(23)^{N_L}2123 x^{-1}$. 

        In the case $m=3$, the standard word of $(23)^{a_{L,i}+1}(21)$ in \eqref{x-1sabx} is equal to $w_2=31$ or $w_1=21$ since $a_{L,i}=1$ or $2$, and the word $w_2 w_1$ is not standard. The sequence $\bc_L$ cannot be of type $+$ or $0$ by Corollary \ref{cor-cannot}. If $\bc_L$ is of type $=$, the expression in the lemma does not have $w_2 w_1$ as a subword (and the argument in the previous paragraph is valid). If $\bc_L$ is of type $-$, the expression in the lemma is not standard since it has $w_2 w_1$ as a subword. This case will be considered in Lemma \ref{lem: m=3 odd L>1}.

\end{proof}

\subsection{Exceptional cases} \label{reductions-5}
In this subsection, we deal with exceptional cases in which the expressions in Lemmas \ref{lem:1standodd} and \ref{lem:generalandodd} are not standard. These cases are restricted to specific conditions with $m=3,4$ or $5$.

	\begin{lem}[$m=3$; level $1$] \label{lem:m=3 level1}
	Assume that the level of $[a,b]$ is $1$ and suppose that 
 $m=3$, $N_1=2$ and $\bc_1$ is of type $-$.

(1) If $N_2> 1$, then 
	the following expression of $s^{a \times b}$ is $\mathcal{S}(3)$-standard:
\begin{equation*}
    w_1[(w_2)^{a_{2,1}-1}w_3\cdots(w_2)^{a_{2,d_2-1}-1}w_3](w_2)^{a_{2,d_2}}, 
	\end{equation*}
	where $w_1=21, w_2=31, w_3=3212$ and the expression inside $[\ ]$ is void if $d_{2}=1$.

(2) If $N_2 =1$, then
	the following expression of $s^{a \times b}$ is $\mathcal{S}(3)$-standard:
\begin{equation*}	\begin{cases}
    w_1 w_3(w_2w_3)^{a_{3,1}-1}[w_6(w_2w_3)^{a_{3,2}-1}\cdots w_6(w_2w_3)^{a_{3,d_3}-1}]w_2^2 &\text{for $\bc_2$ of type $+$,}\\
	21321(w_4)^{a_{3,1}-1}[w_5(w_4)^{a_{3,2}}\cdots w_5(w_4)^{a_{3,d_3}}]2 3131 &\text{for $\bc_2$ of type $-$,}\\ 
	w_1 w_3 w_2^2 &\text{for $\bc_2$ of type $=$,}
	\end{cases}
	\end{equation*}
	where $w_1=21, w_2=31, w_3=3212, w_4=3231, w_5=231321$, $w_6=3132132312$ and the expression inside $[\ ]$ is void if $d_{3}=1$.
\end{lem}

	\begin{example} \label{exa-[17,7]}
	Suppose $m=3$ and $[a,b]=[17,7]$. Then the level of $[a,b]$ is 1. Since $\frac{17}{7}=2+\frac{3}{7}<2+\frac{1}{2}$, we get $N_1=2$ and $\bc_1$ of type $-$. One can check $\bc_1=(3,2,3,2,3,2,2)$. By Proposition \ref{lem:notreduced} or by definition of $s^{a \times b}$, we obtain
	 \begin{equation*}
	 \begin{split}
	 s^{a\times b}=&(23)^3(21)(23)^2(21)(23)^3(21)(23)^2(21)(23)^3(21)(23)^2(21)(23)^2(21)\\
	 =&213\,\underline{121}\,3\,\underline{121}\,3131=21321\,\underline{232}\,123131=21321323123131,
	\end{split}
	\end{equation*}
where the underlined subwords are replaced using relations $121=212$ and $232=323$.

	On the other hand, we have $\bc_2=(1,1,2)$, $N_2=1$ and $\bc_2$ of type $-$. It is clear that $\bc_3=(2)$.
	By Lemma \ref{lem:m=3 level1} (2), 
	\[ s^{a\times b}=21321(w_4)^123131=21321323123131. \]
Thus we get the same standard word.
	\end{example}

	\begin{proof}[Proof of Lemma \ref{lem:m=3 level1}]
Since \[ (31)^{a_{2,i}}(21) = (31)^{a_{2,i}-1}3121=(31)^{a_{2,i}-1}3212 \quad \text{ for } 1 \le i \le d_2,\] we obtain from \eqref{2131a21}
	\begin{align}	
	s^{a\times b}=&(21)(31)^{a_{2,1}}(21)(31)^{a_{2,2}}\cdots (21)(31)^{a_{2,d_2}} \nonumber \\
	=&(21)(31)^{a_{2,1}-1}(3212)(31)^{a_{2,2}-1}(3212)\cdots(31)^{a_{2,d_2-1}}(3212)(31)^{a_{2,d_2}} \nonumber \\
	=&w_1(w_2)^{a_{2,1}-1}w_3\cdots(w_2)^{a_{2,d_2-1}-1}w_3(w_2)^{a_{2,d_2}}.
	\label{eqn-sp}
	\end{align}
Suppose $N_2>1$. It implies that $a_{2,i}>1$ for all $1\leq i \leq d_2$, and the expression \eqref{eqn-sp} is standard.

Now suppose $N_2=1$.
	If $\bc_2$ is of type $+$, we have
\[ \bc_2 =( 1, \underbrace{2, \dots, 2}_{a_{3,1}\text{ times}}, 1, \underbrace{2, \dots, 2}_{a_{3,2}\text{ times}}, \dots, 1, \underbrace{2, \dots, 2}_{a_{3,d_3}\text{ times}}) \] by Lemma \ref{lem:a_{n,i}} (3). We start from \eqref{eqn-sp} to see that $s^{a \times b}$ is equal to
	\begin{equation*}
	\begin{split}
	&(21)(31)^{a_{2,1}-1}(3212)(31)^{a_{2,2}-1}(3212)\cdots(31)^{a_{2,d_2-1}-1}(3212)(31)^{a_{2,d_2}-1}(31)\\
	=&(21)(321231)^{a_{3,1}}(3212)(321231)^{a_{3,2}}(3212)\cdots (321231)^{a_{3,d_3}}(31)\\
	=&(213212)[(313212)^{a_{3,1}-1}(31321\underline{232}12)(313212)^{a_{3,2}-1}(31321\underline{232}12)\cdots(313212)^{a_{3,d_3}-1}](3131)\\
	=&(213212)[(313212)^{a_{3,1}-1}(31321{323}12)(313212)^{a_{3,2}-1}(31321{323}12)\cdots(313212)^{a_{3,d_3}-1}](3131)\\
	=&w_1 w_3(w_2w_3)^{a_{3,1}-1}w_6(w_2w_3)^{a_{3,2}-1}\cdots w_6(w_2w_3)^{a_{3,d_3}-1}(w_2)^2,
	\end{split}
	\end{equation*}
where the underlines are put to indicate the replacements $232=323$.
	The remaining cases can be proven similarly.
	\end{proof}
Recall that we set \[ x= (132)^{\lfloor \frac{L-2}2 \rfloor}1 \quad \text{ and } \quad x^{-1}= 1(231)^{\lfloor \frac{L-2}2 \rfloor}. \]
Now we present the remaining cases in Lemmas \ref{lem: m=3 odd L>1} and \ref{lem:m=3,4,5 even L>1} below. Since the proofs of these lemmas are similar to that of Lemma \ref{lem:m=3 level1}, we omit the proofs.

	\begin{lem} [$m=3$; odd $L\ge 3$] \label{lem: m=3 odd L>1}
	Assume that the level $L$ of $[a,b]$ is $\ge 3$ and odd. 
	Suppose $m=3$, $N_L=1$ and $\bc_L$ is of type $-$. 

(1) If $N_{L+1}\neq 1$, the $\mathcal S(3)$-standard word of $s^{a\times b}$ is equal to
		\begin{equation*}
	x31[{u_3}^{a_{L+1,1}-1}u_4\cdots {u_3}^{a_{L+1,d_{L+1}-1}-1}u_4]{u_3}^{a_{L+1,d_{L+1}}}23x^{-1}	,
\end{equation*}
	where $u_3=31$, $u_4=3212$ and the expression inside $[\ ]$ is void if $d_{L+1}=1$. 

(2) If $N_{L+1}= 1$, the $\mathcal S(3)$-standard word of $s^{a\times b}$ is equal to  
	\begin{equation*}
	\begin{cases}
	x3132{u_5}^{a_{L+2,1}-1}[u_6{u_5}^{a_{L+2,2}}\cdots u_6{u_5}^{a_{L+2,d_{L+2}}}]u_723x^{-1} &\text{for $\bc_{L+1}$ of type $-$,}\\
	x3132[{u_6}^{a_{L+2,1}}u_5{u_6}^{a_{L+2,2}}\cdots u_5]{u_6}^{a_{L+2,d_{L+2}}-1}u_723x^{-1} &\text{for $\bc_{L+1}$ of type $+$,}\\
	x3132u_723x^{-1} &\text{for $\bc_{L+1}$ of type $=$},
	\end{cases}
	\end{equation*}
	where $u_5=1323$, $u_6=123132$, $u_7=123131$ and the expression inside $[\ ]$ is void if $d_{L+2}=1$.
	\end{lem}

	\begin{lem} [$m=3,4,5$; even $L\ge 2$] \label{lem:m=3,4,5 even L>1}
	Assume that the level $L$ of $[a,b]$ is $\ge 2$ and even. Suppose $m=3,4,5$ and $N_L=m-2$. 
	Let  \[ v_3=31, \quad v_4=3231, \quad v_5=2321 \quad \text{ and }\quad v_6=323231 .\]
	Then the $\mathcal S(m)$-standard word of $s^{a\times b}$ is given by the following:
	\begin{equation*}
	\text{for $m=3$, \quad }
	\begin{cases}
	xv_4^{a_{L+1,1}}[v_3v_4^{a_{L+1,2}-1}\cdots v_3v_4^{a_{L+1,d_{L+1}}-1}]23x^{-1} &\text{if $\bc_L$ is of type $-$,}\\
	xv_423x^{-1} &\text{if $\bc_L$ is of type $=$;}
	\end{cases}
	\end{equation*}
	 \begin{equation*}
	\text{for $m=4$, \quad }
	\begin{cases}
	xv_5^{a_{L+1,1}}[v_4v_5^{a_{L+1,2}-1}\cdots v_4v_5^{a_{L+1,d_{L+1}}-1}]23x^{-1} &\text{if $\bc_L$ is of type $-$,}\\
	x v_5 23 x^{-1}  &\text{if $\bc_L$ is of type $=$,}\\
	y 1213 x^{-1} &\text{if $\bc_L$ is of type $0$;}
	\end{cases}
	\end{equation*}
	\begin{equation*}
	\text{for $m=5$, \quad }
	\begin{cases}
	xv_5^{a_{L+1,1}}[v_6v_5^{a_{L+1,2}-1}\cdots v_6v_5^{a_{L+1,d_{L+1}}-1}]23x^{-1} &\text{if $\bc_L$ is of type $-$,}\\
	xv_523x^{-1} &\text{if $\bc_L$ is of type $=$,}\\
	x2123x^{-1} &\text{if $\bc_L$ is of type $0$.}\\
	\end{cases}
	\end{equation*}
Here the expression inside $[\ ]$ is void if $d_{L+1}=1$.
	\end{lem}

	\begin{remark}
	By Corollary \ref{cor-cannot} (1) the sequences $\bc_L$ cannot be of type $+$ in Lemma \ref{lem:m=3,4,5 even L>1}. Moreover, when $m=3$, the type of $\bc_{L}$ cannot be $0$ by Corollary \ref{cor-cannot} (2). Thus all the possible cases are covered.
\end{remark}

\begin{example} \label{exa-[13,5]} 
    (1)	 Suppose $m=3$ and $[a,b]=[13,5]$. Since $\bc_1=(3,3,2,3,2)$ and $\bc_2=(2,1)$, we obtain from Proposition \ref{lem:notreduced}
\begin{equation*}
	  \begin{split}
	  s^{a\times b}=&(23)^3(21)(23)^3(21)(23)^2(21)(23)^3(21)(23)^2(21)\\
	  =&(21)^2(31)(21)(31)=12321231=13231231,
	  \end{split}
	  \end{equation*}
where we use the relations $121=212$, $232=323$ and $22=e$.
 On the other hand, $L=2$, $N_2=1$, $\rho_2= \frac 1 2$ and $\bc_2$ is of type $=$. By Lemma \ref{lem:m=3,4,5 even L>1}, we have 
\[ s^{a \times b} = x v_4 23 x^{-1}= 13231231 ,\] which is the same standard word.

    (2) Let $m=4$ and $[a,b]=[85,23]$. Then \begin{align*} \bc_1&=(
    4, 4, 4, 3, 4, 4, 3, 4, 4, 3, 4, 4, 4, 3, 4, 4, 3, 4, 4, 3, 4, 4, 3),\\ \bc_2&=(3, 2, 2, 3, 2, 2, 2), \quad \bc_3=(2, 3), \end{align*}
    and $N_1=3, N_2=2, \rho_1=\frac {16}{23}, \rho_2=\frac 2 7$. The level $L$ is equal to $2$ and the sequence $\bc_2$ is of type $-$. By Lemma \ref{lem:m=3,4,5 even L>1}, we obtain
    \begin{equation} \label{st-5716} s^{a \times b} = x v_5^2 v_4 v_5^2 23 x^{-1} = 1 (2321)^2 (3231) (2321)^2 231 .\end{equation}  One can check that the initial word of $s^{a \times b}$ indeed reduces to the standard word in \eqref{st-5716}.
\end{example}
	 
\subsection{Proof of Theorem \ref{thm-main}} \label{sub-pm}
\begin{proof}
Suppose $[a,b] \neq [c,d]$ with $a \ge b$ and $c \ge d$. If $[a,b]$ and $[c,d]$ are both real, then $s^{a \times b} \neq s^{c \times d}$ by Lemma \ref{lem-real}. 
Suppose that $[a,b]$ and $[c,d]$ are both imaginary. We will show that $s^{a\times b} \neq s^{c\times d}$ by comparing $\mathcal{S}(m)$-standard words given in Lemmas \ref{lem:1standodd}, \ref{lem:generalandodd}, \ref{lem:m=3 level1}, \ref{lem: m=3 odd L>1} and \ref{lem:m=3,4,5 even L>1}. 

Let the levels of $[a,b]$ and $[c,d]$ be $L$ and $L'$, respectively. 
Without loss of generality, we may assume $L \le L'$. Write $g=\lceil \frac {L-2} 2 \rceil$. It is enough to show $(231)^g s^{a \times b} (132)^g \neq (231)^g s^{c \times b} (132)^g$. By Corollary \ref{cor-red} (2), an expression of $(231)^g s^{a \times b} (132)^g$ is equal to $s^{\tilde a \times \tilde b}$ where $[\tilde a , \tilde b]$ has level $L-2g = 1$ or $2$.
Similarly, $(231)^g s^{c \times d} (132)^g$ is equal to $s^{\tilde c \times \tilde d}$ where $[\tilde c , \tilde d]$ has level $L'-2g$. 
Consequently, we may assume that $[a,b]$ has level $L=1$ or $2$ and $[c,d]$ has level $L' \ge L$, and it is sufficient to prove $s^{a \times b } \neq s^{c \times d}$.

We summarize consequences of Lemmas \ref{lem:1standodd}, \ref{lem:generalandodd}, \ref{lem:m=3 level1}, \ref{lem: m=3 odd L>1} and \ref{lem:m=3,4,5 even L>1} in what follows. 
 Let $[e,f]$ be an arbitrary imaginary positive reduced root of level $L''$.
\begin{itemize}
\item If $L''=1$, the standard word of $s^{e \times f}$ starts with 
one of 
\begin{equation} \label{stw-1}  2131, \quad 2132,\quad 2321,  \quad 2323,  \quad 31,\quad 3231, \quad 3232, \end{equation} as one can see from  Lemmas \ref{lem:1standodd} and \ref{lem:m=3 level1}.

\item If $L''=2$, the standard word of $s^{e \times f}$ starts with
one of \begin{equation} \label{stw-2} 12121, \quad 12123, \quad 12321, \quad 13231 , \quad 21212, \quad 21213\end{equation}  from  Lemmas \ref{lem:generalandodd} and \ref{lem:m=3,4,5 even L>1}.
Note that $2131$ cannot occur as having $\bc_2$ of type $0$ and $N_2=1$ is impossible by Corollary \ref{cor-cannot} (2). Likewise $12323$ cannot occur because having $\bc_2$ of type $+$ and $N_L=m-2$ is impossible by Corollary \ref{cor-cannot} (1).  

\item If $L''=2g+1 \ge 3$, the standard word of $s^{e \times f}$ starts with one of 
\begin{equation} \label{stw-3}  x2323,  \quad x 31, \quad x 3231, \quad x 3232 \end{equation} from
 Lemmas \ref{lem:generalandodd} and \ref{lem: m=3 odd L>1}, where $x= (132)^{\lfloor \frac{L''-2}2 \rfloor}1$. Note that $ x2321$ cannot occur since having $N_L=1$ and $\bc_L$ of type $0$ is impossible by Corollary \ref{cor-cannot} (2).  

\item If $L''=2g+2 \ge 4$, the standard word of $s^{e \times f}$ starts with one of 
\begin{equation} \label{stw-4}  y 1212, \quad y 1213,  \quad y 13, \quad  y 2121, \quad y 2123, \quad x2121, \quad x 2123,  \quad x 2321, \quad x 3231, \end{equation} from 
 Lemmas \ref{lem:generalandodd} and \ref{lem:m=3,4,5 even L>1}, where $x= (132)^{\lfloor \frac{L''-2}2 \rfloor}1$ and $y=(132)^{\frac{L''-4}2}13$. 
\end{itemize}

 First assume $L < L'$. If $L=1$, then none of the words in \eqref{stw-1} appears as a starting word in \eqref{stw-2}, \eqref{stw-3} and \eqref{stw-4}. Thus the standard word of $s^{a\times b}$ must be different from that of $s^{c \times d}$, and hence $s^{a \times b} \neq s^{c \times d}$. If $L=2$ and $L' \ge 4$, then
none of the words in \eqref{stw-2} appears as a starting word in \eqref{stw-3} and \eqref{stw-4}, and we obtain $s^{a\times b} \neq s^{c \times d}$. If $L=2$ and $L'=3$, then
$13231$ is common in \eqref{stw-2} and \eqref{stw-3}. However, if $L=2$, a standard word of $s^{a\times b}$ starts with $13231$ only when $m=3$; now, if $L'=3$ and $m=3$, no standard word actually starts with $13231$ by Lemmas \ref{lem:generalandodd} and \ref{lem: m=3 odd L>1}.
Thus the standard word $s^{a \times b}$ is different from that $s^{c \times d}$ in this case, and we have $s^{a \times b} \neq s^{c \times d}$. 

Next assume that $L=L'=1$. Let $N_k$, $\bc_k$ and $\epsilon_k$ be defined for $[a,b]$ as in Definition \ref{def:a_n}, where $\epsilon_k$ denotes the type of $\bc_k$, and use notations $N'_k$, $\bc'_k$ and $\epsilon'_k$ for $[c,d]$. One can check that the standard words in Lemmas \ref{lem:1standodd} and \ref{lem:m=3 level1} are all different for each $m$.  If $m \neq 3$, the standard words of $s^{a\times b}$ and $s^{c \times d}$ are determined by $(N_1,\epsilon_1, \bc_2)$ and $(N'_1, \epsilon'_1, \bc'_2)$ respectively by Lemma \ref{lem:1standodd}. Since $[a,b] \neq [c,d]$, we have $(N_1, \epsilon_1, \bc_2) \neq (N'_1,\epsilon'_1, \bc'_2)$  by Lemma \ref{lem:ncninj} and 
 the corresponding standard words are different. Thus $s^{a \times b} \neq s^{c \times d}$. 
If $m =3$, the standard words are determined either by $(N_1, \epsilon_1, \bc_2) $ and $(N'_1, \epsilon'_1, \bc'_2)$, or by $(N_1, \epsilon_1, N_2, \epsilon_2, \bc_3) $ and $(N'_1, \epsilon'_1,N'_2, \epsilon'_2, \bc'_3)$. Since $[a,b] \neq [c,d]$, we have $s^{a \times b} \neq s^{c \times d}$ by Lemmas \ref{lem:ncninj}, \ref{lem:1standodd} and \ref{lem:m=3 level1}.

Finally assume that $L=L'=2$. Similarly, as in the case that $L=L'=1$, one can check that the standard words in Lemmas \ref{lem:generalandodd} (1) and \ref{lem:m=3,4,5 even L>1} are all different for each $m$.   The standard words of $s^{a\times b}$ and $s^{c \times d}$ are determined by $(N_2, \epsilon_2,\bc_3)$ and $(N'_2, \epsilon'_2,\bc'_3)$ respectively, and note that $N_1=N'_1=m-1$ and $\epsilon_1=\epsilon'_1=+$. Since $[a,b] \neq [c,d]$, we have $(N_2, \epsilon_2, \bc_3) \neq (N'_2, \epsilon'_2,\bc'_3)$  by Lemma \ref{lem:ncninj} and 
 the corresponding standard words are different. Thus $s^{a \times b} \neq s^{c \times d}$.

Now suppose that $[a,b]$ is real and $[c,d]$ is imaginary. Then the $\mathcal S(m)$-standard expression of $s^{a \times b}$ is given by Lemma  \ref{sFnFn1} (2) and Lemma \ref{lem-real}, and can be written as
\begin{equation} \label{eqn-rr} 21, \qquad  (132)^l 1 (231)^{l+1} \qquad \text{ or } \qquad (132)^l 131 (231)^{l+1} \quad \text{ for some }l \ge 0 .\end{equation}
Comparing \eqref{eqn-rr} with \eqref{stw-1}, \eqref{stw-2}, \eqref{stw-3} and \eqref{stw-4}, we see that only possibilities occur when $s^{c \times d}$ starts with $x31$ or $y2123$. Further, we compare \eqref{eqn-rr} with the standard words starting with $x31$ or $y2123$ in Lemmas  \ref{lem:generalandodd}, \ref{lem: m=3 odd L>1} and \ref{lem:m=3,4,5 even L>1} and see that $s^{a \times b} \neq s^{c \times d}$ in all the possibilities. This completes the proof. 
\end{proof}

\subsection{Proof of Theorem \ref{thm-in}} \label{sub-in}
\begin{proof}
By Theorem 1.2 in \cite{LL1}, the map $[a,b] \mapsto s([a,b])$ is a surjection from the set of reduced positive roots of $\mathcal H(m)$ onto  the set of rigid reflections of $W(m)$. Thus we have only to prove that the map is an injection.

Suppose that $[a,b]$ and $[c,d]$ are two distinct reduced positive roots of $\mathcal H(m)$. If $a \ge b$ and $c \ge d$, then $s^{a \times b} \neq s^{c \times d}$ by Theorem \ref{thm-main}. It follows from Lemma \ref{sFnFn1} (1) that $s([a,b]) \neq s([c,d])$. 
If $a < b$ and $c <d$, then we also have $s([a,b]) \neq s([c,d])$ since interchanging roles of $1$ and $3$ yields a symmetry to cover this case. If $a \ge b$ and $c <d$, or if $a<b$ and $c \ge d$,  then $s([a,b]) \neq s([c,d])$ by Corollary \ref{cor-agb}. This completes the proof. 
\end{proof}

\medskip

\section{Gr\"obner--Shirshov basis for $W(m)$} \label{GS}

In this section we determine \GS bases for $W(m)$, which are used in the previous sections.


\subsection{\GS basis theory}
We briefly recall the \GS basis theory (or Diamond Lemma). See \cite{Be,Bo,Bo-Sh,KL,KL1} for more details. 
Let $X=\{ x_1, x_2, \cdots \}$ be an alphabet
and let $X^*$ be the free monoid of associative monomials on $X$. We denote the empty monomial by $e$
and the {\em length} of a monomial $u$ by $l(u)$. Thus we have $l(e)=0$.
A well-ordering $\prec$ on $X^*$ is called a
{\it monomial order} if $x \prec y$ implies $axb \prec ayb$ for all $a, b \in X^*$. For two monomials 
$$u=x_{i_1} x_{i_2} \cdots x_{i_k}, \quad
v=x_{j_1} x_{j_2} \cdots x_{j_l} \in X^*,$$
define $u \prec_{\text{deg-lex}} v$
if and only if $k<l$ or $k=l$ and $i_r > j_r$ for the first $r$
such that $i_r \neq j_r$; it is a monomial order on $X^*$ called
the {\it degree lexicographic order}.
We denote the degree lexicographic order on $X^*$ simply by $\prec$.
In particular, we have $x_1 \succ x_2 \succ \dots $.

Let $\mathcal{A}_{X}$ be
the free associative algebra generated by $X$ over a field
$\mathbb{F}$. Given a nonzero element $p \in \mathcal{A}_{X}$, we
denote by $\overline{p}$ the maximal monomial appearing in $p$
under the ordering $\prec$. Thus $p = \alpha \overline{p} + \sum
\beta _i w_i $ with $\alpha , \beta _i \in \mathbb{F}$, $ w_i \in
X^*$, $\alpha \neq 0$ and $w_i \prec \overline{p}$. If $\alpha
=1$, $p$ is said to be {\it monic}.

Let $S$ be a subset of monic elements of $\mathcal{A}_X$, let $J$ be the two-sided ideal of
$\mathcal{A}_{X}$ generated by $S$. Then we say that the algebra $A= \mathcal{A}_X /J$
is {\it defined by $S$}. The images of $p \in \mathcal{A}_X$ in $A$ will also be denoted by $p$.

\begin{definition} \label{def-standard} Given a subset $S$ of monic elements of $\mathcal{A}_X$, a monomial $u \in X^*$ is said to be
{\it $S$-standard} if $u \neq a\overline{s}b$ for any $s \in S$ and $a, b
\in X^*$. Otherwise, the monomial $u$ is said to be {\it $S$-reducible}. \end{definition}

Through inductive steps, every $p \in \mathcal{A}_X$ can be expressed as
\begin{equation} \label{equ-1}
p = \sum \alpha_i a_is_ib_i +  \sum \gamma_k u_k,
\end{equation}
where $\alpha_i, \gamma_k \in \mathbb{F}$, $a_i, b_i,  u_k \in X^*$, $s_i \in S$, $a_i \overline{s_i} b_i \preceq \overline{p}$, $u_k \preceq \overline{p}$ and
$u_k$ are $S$-standard.
The term $\sum \gamma_k u_k$ in the expression (\ref{equ-1}) is
called a {\it standard} (or {\em normal}) form of $p$ with respect to the pair $S$ (and with respect to the monomial order $\prec$). In general,
a standard word is not unique. Nonetheless, it is clear that the set of $S$-standard monomials linearly spans the algebra $A$ defined by $S$.

\begin{definition} A subset $S$ of monic elements of $\mathcal{A}_X$ is a {\it \GS basis} if the set of $S$-standard monomials forms a linear basis of the algebra $A$ defined by $S$.
\end{definition}

Let $p$ and $q$ be monic elements of $\mathcal{A}_{X}$ with leading terms $\overline{p}$ and $\overline{q}$. We define
the {\it compositions} of $p$ and $q$ as follows.

\begin{definition}
($a$) If there exist $a, b, w \in X^*$ such that $\overline{p}a = b\overline{q} = w$ with $l(\overline{p}) > l(b)$,
then the {\bf \it composition of intersection} of $p$ and $q$ with respect to $w$ is defined to be $(p,q)_w = pa -bq$.

($b$) If there exist $a, b, w \in X^*$ such that $b \neq e$, $\overline{p}=a\overline{q}b=w$, then the {\bf \it
composition of inclusion}  of $p$ and $q$ with respect to $w$ is defined to be $(p,q)_w = p - aqb$.
\end{definition}

For $p, q \in \mathcal{A}_{X}$ and $w \in X^*$, we define a {\it
congruence relation} on $\mathcal{A}_{X}$ as follows: $p \equiv q
\mod (S; w)$ if and only if $p -q = \sum \alpha_i a_i s_i b_i$, where $\alpha_i \in \mathbb{F}$,
$a_i, b_i \in X^*$, $s_i \in S$, and $a_i
\overline{s_i} b_i \prec w$.

\begin{definition}
A subset $S$ of monic elements in $\mathcal{A}_X$ is said to be {\it closed under composition} if
$(p,q)_w \equiv 0 \mod (S;w)$ for all $p,q \in S$ and for any $w \in X^*$ whenever the composition $(p,q)_w$ is defined.
\end{definition}

The following is Shirshov's Composition Lemma.

\begin{lem} [\cite{Bo}] 
Let $S$ be a subset of monic elements of $\mathcal{A}_{X}$, and
let $A=\mathcal{A}_{X}/J$ be the associative algebra defined by $S$. Assume that $S$ is closed under composition. If the image of $p \in \mathcal{A}_X$ is trivial in $A$, then the word
$\overline{p}$ is $S$-reducible.
\end{lem}

As a consequence, we obtain:

\begin{thm} [\cite{Be,Bo}] \label{cor-1}
Let $\mathscr S$ be a subset of monic elements in $\mathcal{A}_{X}$. Then the following 
are equivalent {\rm :}
\begin{enumerate}
\item [({\it a})] $\mathscr S$ is a \GS  basis;
\item [({\it b})] $\mathscr S$ is closed under composition;
\item [({\it c})] For each $p \in \mathcal{A}_X$, the standard word of $p$ is unique.
\end{enumerate}
\end{thm}

\subsection{Coxeter groups} Consider a Coxeter group
$$
W=\langle s_1,s_2,...,s_n \ : \   s_1^2=\cdots=s_n^2=e, \ (s_is_j)^{m_{ij}}=e \ (i \neq j) \rangle,
$$ where $m_{ij}\in\{2,3,4,...\} \cup \{ \infty \}$. 
Let $X=\{ s_1, s_2, \dots , s_n \}$. Then $X^*$ has the degree lexicographic order $\prec$ defined in the previous subsection. In particular, we have 
\[ s_1 \succ s_2 \succ \cdots \succ s_n. \]
Let $S$ be the set of relations:
\begin{align*}
s_i^2&-e &&\text{for } i=1,2, \dots , n,\\
(s_is_j)^{m_{ij}/2} &- (s_js_i)^{m_{ij}/2} && \text{if $s_i \succ s_j$ and $m_{ij}$ is even},\\
(s_is_j)^{\lfloor m_{ij}/2 \rfloor}s_i &- (s_js_i)^{\lfloor m_{ij}/2 \rfloor}s_j && \text{if $s_i \succ s_j$ and $m_{ij}$ is odd}. 
\end{align*}
As in the previous subsection, let $J$ be the ideal of $\mathcal A_X$ generated by $S$, and $A=\mathcal A_X/J$ be the algebra defined by $S$. Then $A$ is nothing but the group algebra $\mathbb F[W]$ of $W$. 

If $S$ is not a \GS basis, we extend $S$ by putting all nontrivial compositions into $S$, and denote the resulting set of relations by $S^{(1)}$. If it is a \GS basis, we stop; otherwise, we extend $S^{(1)}$ in the same way to obtain $S^{(2)}$ and continue the process. If this process terminates at $S^{(N)}$ for some $N$, we obtain a \GS basis $\mathscr S:=S^{(N)}$ for $A=\mathbb F[W]$. 
By abusing language, we also call 
 $\mathscr S$  a {\em \GS basis} for $W$.

It follows from the construction that every element in $\mathscr S$ is of the form $u-v$ with $u, v \in X^*$, and the identity $u=v$ is valid in the group $W$.  Consequently, an element in $w \in W$ can be written uniquely into an $\mathscr S$-standard monomial using the identities $u=v$ in $W$ for $u-v \in \mathscr S$. 

\subsection{The groups $W(m)$}
In this subsection, for each $m \ge 3$, we will compute a \GS basis for $W(m)$. We need to separate two cases according to the parity of $m$. In the proofs, we simply write $u=v$ for $u \equiv v \mod (S'; w)$ where $S'$ and $w$ are clear from the context.

{(i) Assume that $m=2k-1$, $k \ge 2$.}
Let $X=\{ s_1, s_2, s_3 \}$. The set $S$ of the defining relations are given by
\begin{align}
s_1^2&-e, \label{eqn-ss1} \\s_2^2&-e, \label{eqn-ss2} \\s_3^2&-e, \label{eqn-ss3} \\
(s_1s_2)^{k-1}s_1 &-(s_2 s_1)^{k-1}s_2, \label{eqn-ss2ss1} \\
(s_2s_3)^{k-1}s_2 &-(s_3 s_2)^{k-1}s_3.  \label{eqn-ss3ss2}
\end{align}

\begin{prop}\label{prop:s(m)odd}
Let $m=2k-1$ for $k \ge 2$. The set $S$ of defining relations \eqref{eqn-ss1}-\eqref{eqn-ss3ss2} is a \GS basis $\mathscr S$ for $W(m)$. That is, we have $\mathscr S=S$ in this case.
\end{prop}

\begin{proof} There are no possible compositions among \eqref{eqn-ss1}, \eqref{eqn-ss2} and \eqref{eqn-ss3}. The composition of \eqref{eqn-ss1} and \eqref{eqn-ss2ss1} is 
\begin{align*}
& \eqref{eqn-ss1} \times s_2(s_1s_2)^{k-2}s_1 - s_1 \times \eqref{eqn-ss2ss1} = s_1 (s_2s_1)^{k-1}s_2 - s_2(s_1s_2)^{k-2}s_1 \\ &= (s_1s_2)^{k-1} s_1s_2 -s_2 (s_1 s_2)^{k-2}s_1= (s_2s_1)^{k-1} s_2s_2 - (s_2 s_1)^{k-1} =0, 
\end{align*}
where we use \eqref{eqn-ss2ss1} and \eqref{eqn-ss2}.
The composition of \eqref{eqn-ss2ss1} and \eqref{eqn-ss1} is
\begin{align*}
\eqref{eqn-ss2ss1} \times s_1 - (s_1s_2)^{k-1}  \times \eqref{eqn-ss1} & = - (s_2s_1)^k  + (s_1s_2)^{k-1} = - s_2(s_2s_1)^{k-1}s_2 + (s_1s_2)^{k-1} \\
 & = - (s_1s_2)^{k-1} + (s_1s_2)^{k-1}  = 0 , 
\end{align*}
where we use \eqref{eqn-ss2ss1} and \eqref{eqn-ss2}.
Similarly, the composition  of \eqref{eqn-ss2} and \eqref{eqn-ss3ss2} and that of \eqref{eqn-ss3ss2} and \eqref{eqn-ss2} are all trivial.

The compositions between \eqref{eqn-ss2ss1} and \eqref{eqn-ss2ss1} are 
\begin{align*}
\eqref{eqn-ss2ss1} \times s_2 (s_1s_2)^{\ell -1} s_1 - (s_1s_2)^{\ell}  \times \eqref{eqn-ss2ss1} &= -(s_2s_1)^{k-1} s_2 s_2 (s_1 s_2)^{\ell -1} s_1 + (s_1 s_2)^\ell (s_2 s_1)^{k-1} s_2 \\ &= - (s_2 s_1)^{k-\ell} s_1 + (s_2s_1)^{k-\ell -1} s_2 =0
\end{align*}
for $1 \le \ell \le k-1$, where we use \eqref{eqn-ss1} and \eqref{eqn-ss2}. Similarly, the compositions between \eqref{eqn-ss3ss2} and \eqref{eqn-ss3ss2} are trivial.

There is no more possible composition. Thus the set of defining relations \eqref{eqn-ss1}-\eqref{eqn-ss3ss2} is closed under composition, and it is a \GS basis by Theorem \ref{cor-1}.
\end{proof}

{(ii) Assume that $m=2k$, $k \ge 2$.}
Let $X=\{ s_1, s_2, s_3 \}$. The set $S$ of the defining relations are given by
\begin{align}
s_1^2&-e, \label{eqn-s1} \\s_2^2&-e, \label{eqn-s2} \\s_3^2&-e, \label{eqn-s3} \\
(s_1s_2)^k& -(s_2 s_1)^k,\label{eqn-s2s1}  \\
(s_2s_3)^k &-(s_3 s_2)^k. \label{eqn-s3s2} 
\end{align}

\begin{prop}\label{prop:s(m)even}
Let $m=2k$ for $k \ge 2$. A \GS basis $\mathscr S$ for $W(m)$ is given by the set consisting of defining relations \eqref{eqn-s1}-\eqref{eqn-s3s2} and one additional relation
\begin{equation} \label{eqn-s3s2s1}
(s_1 s_2)^{k-1} s_1 (s_3s_2)^k - (s_2s_1)^k s_3 (s_2s_3)^{k-1}. 
\end{equation}
\end{prop}

\begin{proof}
There are no non-trivial compositions among \eqref{eqn-s1}, \eqref{eqn-s2} and \eqref{eqn-s3}.
The composition of \eqref{eqn-s1} and \eqref{eqn-s2s1} is
\begin{align*}
\eqref{eqn-s1} \times s_2(s_1s_2)^{k-1} - s_1 \times \eqref{eqn-s2s1} & = s_1 (s_2s_1)^k - s_2(s_1s_2)^{k-1} = (s_1s_2)^k s_1 -s_2 (s_1 s_2)^{k-1} \\
 & = (s_2s_1)^k s_1 -s_2 (s_1 s_2)^{k-1} = 0 , 
\end{align*}
where we use \eqref{eqn-s2s1} and \eqref{eqn-s1}.
The composition of \eqref{eqn-s2s1} and \eqref{eqn-s2} is
\begin{align*}
\eqref{eqn-s2s1} \times s_2 - (s_1s_2)^{k-1} s_1 \times \eqref{eqn-s2} & = - (s_2s_1)^k s_2 + (s_1s_2)^{k-1} s_1= - s_2(s_1s_2)^k + (s_1s_2)^{k-1} s_1\\
 & = - s_2(s_2s_1)^k + (s_1s_2)^{k-1} s_1 = 0 , 
\end{align*}
where we use \eqref{eqn-s2s1} and \eqref{eqn-s2}.
Similarly, the composition  of \eqref{eqn-s2} and \eqref{eqn-s3s2} and that of \eqref{eqn-s3s2} and \eqref{eqn-s3} are all trivial.

The composition between \eqref{eqn-s2s1} and \eqref{eqn-s3s2} is
\begin{align*}
\eqref{eqn-s2s1} \times s_3 (s_2s_3)^{k-1} - (s_1s_2)^{k-1}s_1 \times \eqref{eqn-s3s2}& = (s_1 s_2)^{k-1} s_1 (s_3s_2)^k - (s_2s_1)^k s_3 (s_2s_3)^{k-1}.
\end{align*}
Thus we have obtained a new relation
\begin{equation*} 
(s_1 s_2)^{k-1} s_1 (s_3s_2)^k - (s_2s_1)^k s_3 (s_2s_3)^{k-1},
\end{equation*}
which is the relation \eqref{eqn-s3s2s1}.

The composition between \eqref{eqn-s1} and \eqref{eqn-s3s2s1} is
\begin{align*}
& \eqref{eqn-s1} \times s_2(s_1s_2)^{k-2} s_1 (s_3s_2)^k - s_1 \times \eqref{eqn-s3s2s1}  = s_1 (s_2s_1)^k s_3 (s_2s_3)^{k-1}-s_2 (s_1s_2)^{k-2}s_1(s_3s_2)^k \\
& = (s_1 s_2)^k s_1 s_3 (s_2s_3)^{k-1} - (s_2s_1)^{k-1} (s_3s_2)^k = (s_2s_1)^k s_1 s_3 (s_2s_3)^{k-1} - (s_2s_1)^{k-1} (s_3s_2)^k \\ &= (s_2s_1)^{k-1} (s_2 s_3)^{k}  - (s_2s_1)^{k-1} (s_3s_2)^k =0,
\end{align*}
where we use \eqref{eqn-s2s1}, \eqref{eqn-s1} and \eqref{eqn-s3s2}.
The composition between \eqref{eqn-s3s2s1} and \eqref{eqn-s2} is
\begin{align*}
& \eqref{eqn-s3s2s1} \times s_2 - (s_1 s_2)^{k-1} s_1 (s_3s_2)^{k-1} s_3 \times \eqref{eqn-s2} \\ =& - (s_2 s_1)^k s_3 (s_2s_3)^{k-1} s_2 + (s_1 s_2)^{k-1} s_1 (s_3s_2)^{k-1} s_3 \\  =& - s_2 (s_1 s_2)^{k-1} s_1 (s_3s_2)^k + (s_1 s_2)^{k-1} s_1 (s_3s_2)^{k-1} s_3 \\ =& - s_2 (s_2 s_1)^k s_3 (s_2 s_3)^{k-1} + (s_1 s_2)^{k-1} s_1 (s_3s_2)^{k-1} s_3 \\  =& - (s_1 s_2)^{k-1} s_1 s_3 (s_2 s_3)^{k-1} + (s_1 s_2)^{k-1} s_1 (s_3s_2)^{k-1} s_3 =0, 
\end{align*}
where we use \eqref{eqn-s3s2s1} and \eqref{eqn-s2}.

The compositions between \eqref{eqn-s2s1} and \eqref{eqn-s3s2s1} are, for $1 \le \ell \le k-1$, 
\begin{align*}
& \eqref{eqn-s2s1} \times (s_1s_2)^{\ell -1} s_1 (s_3 s_2)^k - (s_1s_2)^\ell \times \eqref{eqn-s3s2s1} \\ =& - (s_2s_1)^k (s_1 s_2)^{\ell -1} s_1 (s_3 s_2)^k + (s_1 s_2)^\ell (s_2 s_1)^k s_3 (s_2 s_3)^{k-1}  \\
= &- (s_2s_1)^{k-\ell+1}  s_1 (s_3 s_2)^k + (s_2 s_1)^{k-\ell} s_3 (s_2 s_3)^{k-1} \\ = & - (s_2s_1)^{k-\ell}  (s_2 s_3)^k s_2 + (s_2 s_1)^{k-\ell}  (s_3 s_2)^{k-1} s_3 \\ = & - (s_2s_1)^{k-\ell}  (s_3 s_2)^k s_2 + (s_2 s_1)^{k-\ell}  (s_3 s_2)^{k-1} s_3 =0 ,
\end{align*}
where we use \eqref{eqn-s1}, \eqref{eqn-s2} and \eqref{eqn-s3s2}.
The compositions between \eqref{eqn-s3s2s1} and \eqref{eqn-s3s2} are, for $1 \le \ell \le k$,
\begin{align*}
& \eqref{eqn-s3s2s1} \times (s_3 s_2)^{\ell -1} s_3 - (s_1 s_2)^{k-1} s_1 s_3 (s_2 s_3)^{\ell -1} \times \eqref{eqn-s3s2} \\ =& (s_1s_2)^{k-1} s_1s_3 (s_2s_3)^{\ell -1} (s_3 s_2)^k - (s_2s_1)^k s_3 (s_2 s_3)^{k-1} (s_3s_2)^{\ell -1} s_3 \\
=& (s_1s_2)^{k-1} s_1s_3 (s_3 s_2)^{k-\ell+1} - (s_2s_1)^k s_3 (s_2 s_3)^{k-\ell}  s_3 \\
=& (s_1s_2)^{k-1} s_1 (s_2 s_3)^{k-\ell}s_2 - (s_2s_1)^k (s_3 s_2)^{k-\ell}\\
=& (s_1s_2)^{k} (s_3 s_2)^{k-\ell} - (s_2s_1)^k (s_3 s_2)^{k-\ell} =0,
\end{align*}
where we use \eqref{eqn-s2}, \eqref{eqn-s3} and \eqref{eqn-s2s1}.

The compositions between \eqref{eqn-s2s1} and \eqref{eqn-s2s1} are 
\begin{align*}
\eqref{eqn-s2s1} \times  (s_1s_2)^{\ell }  - (s_1s_2)^{\ell}  \times \eqref{eqn-s2s1} &= -(s_2s_1)^{k} (s_1 s_2)^{\ell }  + (s_1 s_2)^\ell (s_2 s_1)^{k}  \\ &= - (s_2 s_1)^{k-\ell}  + (s_2s_1)^{k-\ell }  =0
\end{align*}
for $1 \le \ell \le k-1$, where we use \eqref{eqn-s1} and \eqref{eqn-s2}. Similarly, the compositions between \eqref{eqn-s3s2} and \eqref{eqn-s3s2} are trivial.

There is no more possible composition. Thus the set consisting of relations \eqref{eqn-s1}-\eqref{eqn-s3s2s1} is closed under composition, and it is a \GS basis by Theorem \ref{cor-1}.
\end{proof}

\end{document}